\documentclass[a4paper, 11pt]{article}
\usepackage{amsmath} 
\usepackage{amsthm}
\usepackage{amssymb}
\usepackage{latexsym}
\usepackage{color}
\usepackage{comment}
\usepackage{mathtools}
\usepackage{here}
\numberwithin{equation}{section}
\newtheorem{thm}{Theorem}[section]
\newtheorem{prop}[thm]{Proposition}
\newtheorem{lemm}[thm]{Lemma}

\newtheorem{defn}[thm]{Definition}
\newtheorem{rem}[thm]{Remark}

\newcommand{\BBB}{\mathbb}
\newcommand{\R}{{\BBB R}}
\newcommand{\Z}{{\BBB Z}}

\newcommand{\N}{{\BBB N}}
\newcommand{\C}{{\BBB C}}

\newcommand{\ee}{\mbox{\boldmath $1$}}

\newcommand{\disp}{\displaystyle}

\newcommand{\bvec}[1]{\mbox{\boldmath $#1$}}
\newcommand{\htype}{{\mathbb H}^d_p}
\newcommand{\G}{{\mathbb G}}

\DeclareMathOperator{\esssup}{ess\sup}

\title{Strichartz estimates and its application to the well-posedness of the nonlinear Schr\"odinger equations on H-type groups}
\author{HIroyuki HIRAYAMA and Yasuyuki OKA
}
\date{}
\begin{document}
\maketitle
\begin{abstract}
The aim of this article is to give the well-posedness results
for the Cauchy problem of the nonlinear Schr\"odinger equation with power type nonlinearities on H-type groups. 
To do this, we prove the dispersive estimate and Strichartz estimate.  
Although these estimates are given by Hierro (2005), its complete proofs cannot be find. 
We correct the statement of these estimates, give the proofs, and apply to the nonlinear problem. 
Our well-posedness results are an improvement of the previous result by Bruno {\it et al}. 
\end{abstract}
\section{Introduction}
\noindent

We consider the Cauchy problem of the nonlinear Schr\"odinger equation (NLS for short)
\begin{equation}\label{HNLS}
\begin{cases}
i\partial_tu+\mathcal{L}u=\mu |u|^{\alpha -1}u,\ \ t>0,\ g\in \htype,\\
u(0,g)=u_0(g),\ \ g\in \htype,
\end{cases}
\end{equation}
where $\htype$ denotes the H-type group with the center dimension $p$ 
and the homogeneous dimension $N=2d+2p$,
$\mathcal{L}$ is a sub-Laplacian, and $\mu \in \C\setminus \{0\}$ is a constant. 
We give the initial data $u_0$ in 
the $L^2$-based Sobolev spaces $H^s(\htype)$. 
The definition of H-type group, sub-Laplacian, 
and Sobolev spaces on H-type group will be given in the next section. 
The typical example of H-type group is the Heisenberg group $\mathbb{H}^d_1$, 
which is H-type group with the center dimension $p=1$. 

On the Euclidean space $\R^d$, 
there are a lot of works for the well-posedness of NLS. 
In particular, the well-posedness on $H^s(\R^d)$ for $s\ge \max\{0,s_c\}$ 
was proved in \cite{Dinh17} (see also \cite{CW90}) under the assumption 
\begin{equation}\label{smooth_cond_nonl}
\alpha\ \text{is\ an\ odd\ integer},\ \text{or}\ \alpha\ge \lceil s\rceil+1,
\end{equation}
where $s_c=\frac{d}{2}-\frac{2}{p-1}$ is the scaling critical Sobolev exponent, 
and $\lceil s\rceil=\min\{n\in \Z\mid n\ge s\}$. 
The assumption {\rm (\ref{smooth_cond_nonl})} is 
a regularity condition for the nonlinear term $f(u)=\mu |u|^{\alpha-1}u$ 
to guarantee the continuous dependence of solution on initial data on Sobolev space. 
When $0\le s \le 2$ and $s<\frac{d}{2}$, 
the assumption {\rm (\ref{smooth_cond_nonl})} was removed 
(see, \cite{CFH11}, \cite{GV79}, \cite{Kato87}, \cite{Kato89}, \cite{Tsu87}, \cite{UW12}). 

The Strichartz estimate is a very useful tool to obtain the well-posedness 
of nonlinear dispersive equation such as NLS. 
On $\R^d$, it is known that the Strichartz estimate
\[
\|u\|_{L^q_t(\R;L^r_x(\R^d))}\le C\|u_0\|_{L^2_x(\R^d)}
\]
with admissible pair $(q,r)\in [2,\infty]^2$ such as
\[
\frac{2}{q}=d\left(\frac{1}{2}-\frac{1}{r}\right),\ \ (q,r,d)\ne (2,\infty,2)
\]
holds for the solution $u=u(t,x)$ to the linear Schr\"odinger equation 
(see, \cite{CW90}, \cite{GV92}, \cite{KT98}, and \cite{Ya87}). 
The dispersive estimate, which is the time decay estimate such as 
\begin{equation}\label{dip_est_intro}
\|u(t)\|_{L^p}\le C|t|^{-\alpha}\|u_0\|_{L^{p'}}
\end{equation}
with $\frac{1}{p}+\frac{1}{p'}=1$ for the solution $u$ to the linear Schr\"odinger equation, 
plays an important role to obtain the Strichartz estimates. 
However, in \cite{Bah}, Bahouri {\it et al.} showed that the existence 
of a function $u_0$ in Schwartz class on the Heisenberg group $\mathbb{H}^d_1$  such that the 
solution $u$ to the linear Schr\"odinger equation with initial data $u_0$ satisfies
\begin{equation}\label{heis_sol_trans}
u(t,z,\eta)=u_0(z,\eta +4td)\ (g=(z,\eta)\in \mathbb{H}^d_1).
\end{equation}
This means that the dispersive estimate such as {\rm (\ref{dip_est_intro})} 
does not hold generally on $\mathbb{H}^d_1$ because {\rm (\ref{heis_sol_trans})} implies 
$\|u(t)\|_{L^p}=\|u_0\|_{L^p}$ for any $t\in \R$.  
On the other hand, in the case $p>1$, the time decay estimate 
for the Schr\"odinger equation on $\htype$ 
is given by Hierro in \cite{Hi2005} (see, Proposition~\ref{disp_hierro} below). 
Therefore, it is expected that the Strichartz estimate 
for the linear Schr\"odinger equation can be obtained for 
H-type group $\htype$ with center dimension $p>1$. 
Actually, we will give the Strichartz estimate below ({\rm Theorem~\ref{H-Strichartz}). 
To prove the Strichartz estimate, we also give the dispersive estimate below ({\rm Proposition~\ref{disp_linf_est}}). 
For the Heisenberg group $\mathbb{H}^d_1$, 
in \cite{BBG21}, Bahouri {\it et al.} proved the the Strichartz estimate 
for the Schr\"odinger equation on $\mathbb{H}^d_1$ with radial initial data. 
To obtain the Strichartz estimate, 
the authors of \cite{BBG21} used the Fourier restriction theory instead of the dispersive estimate. 
The radial assumption for initial data is removed in \cite{BFpre}. 
Additionally in \cite{BFpre}, the authors proved the Strichartz estimate 
on $\htype$ with $p>1$ by using the Fourier restriction approach. 
We will compare the result in \cite{BFpre} with our Strichartz estimate
 (see, Remark~\ref{bari_compare_our} below). 

%
%
%
 
The aim of this paper is to give the well-posedness results for {\rm (\ref{HNLS})}. 
To guarantee the smoothness of nonlinear term, 
we assume the assumption {\rm (\ref{smooth_cond_nonl})}, 
which will be used in the fractional Leibniz rule (Proposition~\ref{FLR} below). 
Our first result is the following. 
\begin{thm}\label{th1}
Let $d, p\in \N$ with $p\ge 2$, 
\[
\max\left\{\frac{N-p+1}{2},\frac{N-2}{2}\right\}\le s<\frac{N}{2},\ \ 
s>\min\left\{\frac{N-p+1}{2},\frac{N-2}{2}\right\},
\]
and $1<\alpha <1+\frac{4}{N-2s}$.  
We also assume $\alpha \ge \lceil s\rceil$ if $\alpha$ is not an odd integer. 
Then, {\rm (\ref{HNLS})} is locally well-posed in $H^s(\htype)$ 
in the following sence:
\begin{itemize}
\item[{\rm (i)}] For any $u_0\in H^s(\htype)$, 
there exist $T=T(\|u_0\|_{H^s})>0$ 
and a solution $u\in C([0,T];H^s(\htype))$, 
which solution is unique in 
the suitable space $X^s_T$\ $($The definition of $X^s_T$ will be given in Section~{\rm \ref{sec_wp}}$)$. 
\item[{\rm (ii)}] For any $\eta>0$, the data-to-solution map 
$u_0\mapsto u$ is continuous from $H^s(\htype)$ to $C([0,T];H^{s-\eta}(\htype))$. 
Furthermore, if {\rm (\ref{smooth_cond_nonl})} is satisfied, then we can take $\eta=0$. 
Namely, the data-to-solution map 
$u_0\mapsto u$ is continuous from $H^s(\htype)$ to $C([0,T];H^{s}(\htype))$. 
\end{itemize}
\end{thm}
\begin{rem}
The condition $1<\alpha <1+\frac{4}{N-2s}$ is equivalent to $s>s_c$, 
where
\[
s_c:=\frac{N}{2}-\frac{2}{\alpha -1}
\]
is the scaling critical exponent for {\rm (\ref{HNLS})}. 
Therefore, if we put
\[
s_*:=\max\left\{\frac{N-p+1}{2},\frac{N-2}{2},s_c\right\},
\]
then Theorem~\ref{th1} says that {\rm (\ref{HNLS})} with $\alpha >1$ 
is locally well-posed in $H^s(\htype)$
for $s_*<s<\frac{N}{2}$ under the assumption {\rm (\ref{smooth_cond_nonl})}. We note that
\[
s_*=\begin{cases}
\frac{N-1}{2}&{\rm if}\ \ 1<\alpha <5\\
s_c&{\rm if}\ \ \alpha \ge 5
\end{cases}
\ \ \ \ {\rm when}\ p=2
\]
and 
\[
s_*=\begin{cases}
\frac{N-2}{2}&{\rm if}\ \ 1<\alpha <3\\
s_c&{\rm if}\ \ \alpha \ge 3
\end{cases}
\ \ \ \ {\rm when}\ p=3,4,5,\cdots. 
\]
In particular, the local well-posedness of {\rm (\ref{HNLS})} 
in $H^s(\mathbb{H}^d_p)$ for $s>s_c$ can be obtained 
when $p=2$, $\alpha \ge 5$ or $p\ge 3$, $\alpha \ge 3$. 
Note that the relation $\frac{N-2}{2}=\frac{N-p+1}{2}$ is equivalent to $p=3$. 
Therefore, Theorem~\ref{th1} also says that the local well-posedness 
of {\rm (\ref{HNLS})} with $\alpha >1$ in $H^s(\htype)$ 
can be obtained for $s=s_*=\frac{N-1}{2}$ when $p=2$, $1<\alpha <5$, 
and for $s=s_*=\frac{N-2}{2}$ when $p=4,5,\cdots$, $1<\alpha <3$. 
\end{rem}

We also treat the scaling critical cases. The second result is the following. 
\begin{thm}\label{th_cri}
Let $d, p\in \N$ with $p\ge 2$. 
We assume one of the following conditions. 
\begin{equation}\label{cri_cond}
{\rm (a)}\ p=2,\ \alpha \ge 5\ \ \ 
{\rm (b)}\ p=3,\ \alpha >3\ \ \ 
{\rm (c)}\ p=4,5,6,\cdots,\ \alpha \ge 3.
\end{equation}
We also assume $\alpha \ge \lceil s\rceil$ if $\alpha$ is not an odd integer. 
Then, {\rm (\ref{HNLS})} is globally well-posed in $H^{s_c}(\htype)$ for small initial data 
in the following sence:
\begin{itemize} 
\item[{\rm (i)}] There exists $\epsilon >0$, 
such that 
for any $u_0\in H^{s_c}(\htype)$ with $\|u_0\|_{H^{s_c}}<\epsilon$, 
there exist a solution $u\in C([0,\infty);H^{s_c}(\htype))$, 
which solution is unique in 
the suitable space $X^{s_c}$. $($The definition of $X^{s_c}$ will be given in Section~{\rm \ref{sec_wp}}.$)$
\item[{\rm (ii)}] For any $\eta>0$, the data-to-solution map 
$u_0\mapsto u$ is continuous  from $B_{\epsilon}(H^{s_c}(\htype))$ to $C([0,\infty);H^{s_c-\eta}(\htype))$, 
where $B_{\epsilon}(H^{s_c}(\htype))$ denotes the set of all $f\in H^{s_c}(\htype)$ satisfying $\|f\|_{H^{s_c}}<\epsilon$. 
Furthermore, if {\rm (\ref{smooth_cond_nonl})} is satisfied, then we can take $\epsilon=0$. 
Namely, the data-to-solution map 
$u_0\mapsto u$ is continuous from $B_{\epsilon}(H^{s_c}(\htype))$ to $C([0,\infty);H^{s_c}(\htype))$. 
\end{itemize}
\end{thm}
For reader's convenience, 
we give a table for the conditions of $p$, $\alpha$, and $s$, 
which arrows the well-posedness of {\rm (\ref{HNLS})} 
under the assumption {\rm (\ref{smooth_cond_nonl})}
by {\rm Theorems~\ref{th1}} and~{\ref{th_cri} 
(see, {\rm Table~\ref{wp_range_s}  below). 
\begin{table}[H]
\label{wp_range_s}
\begin{center}
\begingroup
\renewcommand{\arraystretch}{1.2}
\begin{tabular}{|c|c|c|c|c|c|c|}
\hline
\multicolumn{1}{|c|}{dimension of the center} & \multicolumn{2}{|c|}{$p=2$} & \multicolumn{2}{|c|}{$p=3$} & 
\multicolumn{2}{|c|}{$p=4,5,6,\cdots$}\\
\hline
\multicolumn{1}{|c|}{degree of the nonlinear term} & \multicolumn{1}{|c|}{$1<\alpha <5$} &\multicolumn{1}{|c|}{$\alpha \ge 5$} 
& \multicolumn{1}{|c|}{$1<\alpha \le 3$} &\multicolumn{1}{|c|}{$\alpha > 3$} 
&  \multicolumn{1}{|c|}{$1<\alpha <3$} &\multicolumn{1}{|c|}{$\alpha \ge 3$}\\
\hline
\multicolumn{1}{|c|}{range of the Sobolev index} & \multicolumn{1}{|c|}{$s\ge \frac{N-1}{2}$} &\multicolumn{1}{|c|}{$s\ge s_c$} 
& \multicolumn{1}{|c|}{$s>\frac{N-2}{2}$} &\multicolumn{1}{|c|}{$s\ge s_c$} 
&  \multicolumn{1}{|c|}{$s\ge \frac{N-2}{2}$} &\multicolumn{1}{|c|}{$s\ge s_c$}\\
\hline
\end{tabular}
\caption{Results in {\rm Theorems~\ref{th1}} and~{\ref{th_cri}}}
\endgroup
\end{center}
\end{table}

\begin{rem}
When $p=3$ and $1<\alpha \le 3$, it holds that
\[
\frac{N-p+1}{2}=\frac{N-2}{2}\le s_c. 
\]
Namely, $H^{s}(\mathbb{H}^d_3)$ for $s=\frac{N-2}{2}$ in this case is the intersection of two critical spaces. 
In particular, if $p=3$ and $\alpha =3$, then we have
\[
s_*=\frac{N-p+1}{2}=\frac{N-2}{2}=s_c. 
\]
and $H^{s_c}(\mathbb{H}^d_3)$ is the intersection of three critical spaces. 
Because of such reason, we cannot treat 
the critical cases for $p=3$ and $1<\alpha \le 3$\ $($see, also {\rm Remarks~{\ref{p3_rem_adm}} and~{\ref{critical_adm_rem}}} below$)$.
\end{rem}
\begin{rem}
In {\rm \cite{Bru}}, by using the property of Banach algebra and embedding theorem of Sobolev spaces on stratified Lie groups 
$($or a group more generalized than stratified Lie groups$)$ instead of Strichartz estimates, Bruno {\it et al.} 
showed the unique existence of solutions NLS on stratified Lie groups including Heisenberg groups and H-type groups. However, their results need the strong condition $s>\frac{N}{2}$ for the regularity of initial data. 
Our results, though restricted to H-type groups, are an improvement in that it holds for lower regularity than the condition of the regularity in {\rm \cite{Bru}}.
\end{rem}
\begin{rem}
The assumption $\alpha \ge \lceil s\rceil +1$ 
is for the purpose of obtaining 
the continuity of solution map $H^s(\htype)\ni u_0\mapsto u\in C([0,T);H^s(\htype))$. 
The existence and uniqueness of solution can be obtained 
under the weaker condition $\alpha \ge \lceil s\rceil$ $($see, also {\rm Remark~\ref{ex_uni_rem}}$)$. 
\end{rem}

The key tool for {\rm Theorems~\ref{th1}} and~{\ref{th_cri}  is the Strichartz estimate 
such as
\begin{equation}\label{stri_linsol}
\|u\|_{L^q_tL^r_g}\le C\|u(0)\|_{\dot{H}^{\sigma}}
\end{equation}
for solutions $u=u(t)\in C(\R;H^{\sigma}(\htype))$ to 
the linear Schr\"odinger equation
\[
i\partial_tu+\mathcal{L}u=0.
\]
Let $S_{\rm lin}$ denotes
the set of all solutions $u\in C(\R;H^{\sigma}(\htype))\cap L^q(\R;L^r(\htype))$ to 
the linear Schr\"odinger equation and put
\[
C_{q,r,\sigma}:=\sup\left\{\left.\frac{\|u\|_{L^q_tL^r_g}}{\|u(0)\|_{\dot{H}^{\sigma}}}\ \right|u\in S_{\rm lin}\right\}. 
\]
The Strichartz estimate {\rm (\ref{stri_linsol})} is equivalent to $C_{q,r,\sigma}<\infty$. 
For $\lambda >0$, we define the scaling transformation as
\[
u_{\lambda}(t,g)=u(\lambda^2t,\delta_{\lambda}(g)), 
\]
where $\delta_{\lambda}$ is the dilation function given by
\[
\delta_{\lambda}(g)=(\lambda z,\lambda^2\eta)\ \ (g=(z,\eta)\in \mathbb{H}_p^d). 
\]
We  will also define the dilation function in
{\rm Remark~\ref{charaHtype}} below. 
We note that if $u\in S_{\rm lin}$ holds, 
then $u_{\lambda}\in S_{\rm lin}$ also holds, 
and we have
\[
\begin{split}
\|u_{\lambda}\|_{L^q_tL^r_g}
&=\left\{
\int_{\R}\left(
\int_{\R^{2d}\times \R^p}|u(\lambda^2t,\lambda z,\lambda^2s)|^rdzds
\right)^{\frac{q}{r}}dt
\right\}^{\frac{1}{q}}
=\lambda^{-\frac{N}{r}-\frac{2}{q}}\|u\|_{L^q_tL^r_g}, \\
\|u_{\lambda}(0)\|_{\dot{H}^{\sigma}}
&=\left(
\int_{\R^{2d}\times \R^p}|(-\mathcal{L})^{\frac{\sigma}{2}}(u(0,\lambda z,\lambda^2s))|^2dzds
\right)^{\frac{1}{2}}
=\lambda^{\sigma-\frac{N}{2}}\|u(0)\|_{\dot{H}^{\sigma}}, 
\end{split}
\]
where we used $N=2d+2p$. Therefore, it holds that 
\[
C_{q,r,\sigma}\ge \sup_{\lambda >0}\frac{\|u_{\lambda}\|_{L^q_tL^r_g}}{\|u_{\lambda}(0)\|_{\dot{H}^{\sigma}}}
=\sup_{\lambda >0}\lambda^{-\frac{N}{r}-\frac{2}{q}+\frac{N}{2}-\sigma}\frac{\|u\|_{L^q_tL^r_g}}{\|u(0)\|_{\dot{H}^{\sigma}}}.
\]
This says that the equality
\begin{equation}\label{nece_cond_stri}
\sigma =N\left(\frac{1}{2}-\frac{1}{r}\right)-\frac{2}{q}
\end{equation}
is necessary condition for the Strichartz estimate {\rm (\ref{stri_linsol})}. 
\begin{rem}
In {\rm \cite{Hi2005}}, 
the Strichartz estimate is given under the condition 
\[
\sigma=N\left(\frac{1}{2}-\frac{1}{r}\right)-\frac{1}{q}.
\]
However, the Strichartz estimate {\rm (\ref{stri_linsol})} does not 
holds under this condition because {\rm (\ref{nece_cond_stri})}
is a necessary condition for {\rm (\ref{stri_linsol})}. 
To show the Strichartz estimate, 
the dispersive estimate is also given in {\rm \cite{Hi2005}}. 
But its proof is omitted $($see, {\rm Remark~\ref{disp_rem_hierro}} below$)$. 
We will modify the statement and give the proof of dispersive estimate.  
We will also prove the Strichartz estimate 
under the natural condition {\rm (\ref{nece_cond_stri})}. 
\end{rem}

To give the statement of our Strichartz estimate, 
we define admissible pair. 
\begin{defn}\label{s-admissible}
Let $\sigma \in \R$, $p\in \N$ with $p\ge 2$. 
We say that $(q,r)\in [2,\infty]\times [2,\infty]$ 
with $(q,r)\ne (\infty,\infty)$ be admissible pair 
if the following conditions hold$:$
\begin{equation}\label{sad_cond}
\frac{2}{q}\le (p-1)\left(\frac{1}{2}-\frac{1}{r}\right),\ \ 
(q,r,p)\ne (2,\infty, 3).
\end{equation}
\end{defn}
\begin{rem}\label{endpoint_rem}
If $p>3$, then $(q,r)=(2,\frac{2(p-1)}{p-3})$ becomes 
an admissible pair. 
We call this pair ``end point''. 
If admissible pair $(q,r)$ is not end point, 
then we call this  pair ``non-end point''. 
We note that non-end point $(q,r)$ satisfies 
at least one of
\[
q>2\ \ \ or\ \ \ 
\frac{2}{q}<(p-1)\left(\frac{1}{2}-\frac{1}{r}\right).
\]
\end{rem}
%
%
%
%

We denote the Schr\"odinger semigroup for the sub-Laplacian $\mathcal{L}$ by $\{e^{it\mathcal{L}}\}_{t\in \R}$. 
Namely, $u(t)=e^{it\mathcal{L}}u_0$ means the solution to linear Schr\"odinger equation 
with initial data $u_0$. 
We note that the operator $e^{it\mathcal{L}}$ is unitary on $L^2(\htype)$ (See, also subsection~\ref{SFT_subsec} below). 
Our Strichartz estimate is the following. 
\begin{thm}[Strichartz estimates]\label{H-Strichartz}
Let $d, p\in \N$ with $p\ge 2$. 
\begin{itemize}
\item[{\rm (i)}] Let $(q,r)$ is non-end point admissible pair. 
We put $\sigma =N\left(\frac{1}{2}-\frac{1}{r}\right)-\frac{2}{q}$. 
Then, there exists $C>0$ such that
\begin{equation}\label{stri_hom_est}
\|e^{it\mathcal{L}}u_0\|_{L^q_tL^r_g}
\le C\|u_0\|_{\dot{H}^{\sigma}}
\end{equation}
holds for any $u_0\in \dot{H}^{\sigma}(\htype)$. 
\item[{\rm (ii)}] Let $(q_1,r_1)$ and $(q_2,r_2)$ are non-end point admissible pairs. 
We put $\sigma_k=N\left(\frac{1}{2}-\frac{1}{r_k}\right)-\frac{2}{q_k}$\ $(k=1,2)$. 
Then, there exists $C>0$ such that
\begin{equation}\label{stri_inhom_est}
\left\|\int_0^te^{i(t-t')\mathcal{L}}F(t')dt'\right\|_{L^{q_1}_tL^{r_1}_g}
\le C\|F\|_{L^{q_2'}_t\dot{W}^{\sigma_1+\sigma_2,r_2'}_g}
\end{equation}
holds for any $F\in L^{q_2'}(\R;\dot{W}^{\sigma_1+\sigma_2,r_2'}(\htype))$, 
where $q_2'$ and $r_2'$ denote the conjugate index of $q_2$ and $r_2$ respectively
\end{itemize}
\end{thm}
We note that $\dot{W}^{s,r}(\htype)$ denotes the $L^r$-based homogeneous Sobolev space, 
which will be defined in {\rm Definition~\ref{3_1_1}}. 
We will also define the inhomogeneous Sobolev space $W^{s,r}$. 
\begin{rem}
On the Euclidean space $\R^d$, 
the Strichartz estimate such as {\rm (\ref{stri_hom_est})} without derivative loss $($namely, $\sigma=0)$ can be obtained. 
On the other hand, on the H-type group $\htype$, 
the Strichartz estimate {\rm (\ref{stri_hom_est})} contains derivative loss $\sigma>0$ except  the case $(q,r)=(\infty,2)$. 
This fact makes problem difficult to 
obtain the well-posedness of {\rm (\ref{HNLS})}, 
and we have to find suitable admissible pair. $($See, Lemma~{\rm \ref{adm_exist}}.$)$ 
\end{rem}
\begin{rem}
By the embedding $W^{\sigma,r}\hookrightarrow \dot{W}^{\sigma,r}$ $($and $H^{\sigma}\hookrightarrow \dot{H}^{\sigma})$ for $\sigma \ge 0$, 
We can also obtain
\begin{equation}\label{Stri_inhom_sob_lin}
\|e^{it\mathcal{L}}u_0\|_{L^q_tL^r_g}
\le C\|u_0\|_{H^{\sigma}}
\end{equation}
and 
\begin{equation}\label{Stri_inhom_sob_duam}
\left\|\int_0^te^{i(t-t')\mathcal{L}}F(t')dt'\right\|_{L^{q_1}_tL^{r_1}_g}
\le C\|F\|_{L^{q_2'}_tW^{\sigma_1+\sigma_2,r_2'}_g}
\end{equation}
under the same conditions in {\rm Theorem~\ref{H-Strichartz}}. 
\end{rem}
\begin{rem}\label{p3_rem_adm}
To obtain the well-posedness of {\rm (\ref{HNLS})} in $H^s(\htype)$, 
we will seek admissible pair $(q,r)\in [2,\infty]\times [2,\infty]$ and $\sigma \ge 0$
with $\sigma =N\left(\frac{1}{2}-\frac{1}{r}\right)-\frac{2}{q}$ 
satisfying $q\ge \alpha -1$ and $s-\sigma \ge \frac{N}{r}$. 
For the case $p=3$ and $s=\frac{N-2}{2}$, 
there are no admissible pair satisfying such condition. 
Indeed, $(q,r)=(2,\infty)$ is a necessary condition of
\[
\frac{2}{q}\le 2\left(\frac{1}{2}-\frac{1}{r}\right),\ \ 
\sigma =N\left(\frac{1}{2}-\frac{1}{r}\right)-\frac{2}{q},\ \ 
\frac{N-2}{2}-\sigma \ge \frac{N}{r}
\]
for $(q,r)\in[2,\infty]\times [2,\infty]$. 
But $(2,\infty)$ is not admissible pair when $p=3$. 
\end{rem}
\begin{rem}\label{critical_adm_rem}
$(2,\infty)$ becomes admissible pair 
$($with $\sigma =\frac{N-2}{2})$ when $p>3$, 
and $(4,\infty)$ becomes admissible pair 
$($with $\sigma=\frac{N-1}{2})$ when $p=2$. 
Moreover, if one of the conditions in  {\rm (\ref{cri_cond})} holds, 
then $(q,r)=(\alpha-1,\infty)$ becomes admissible pair 
$($with $\sigma =s_c)$. 
On the other hand when $p=3$ and $\alpha =3$, then $(q,r)=(\alpha -1,\infty)=(2,\infty)$ 
does not become admissible pair. 
\end{rem}
\begin{rem}
For the end point $(q.r)=(2,\frac{2(p-1)}{p-3})$ with $p>3$, 
we don't know whether the Strichartz estimates
such as in {\rm Theorem~\ref{H-Strichartz}} hold or not. 
To consider this problem, 
we need the complicated argument $($cf. {\rm \cite{KT98}}$)$.  
Because the end point Strichartz estimate will not be used
in our main well-posedness results, we don't say anything more about it.  
\end{rem}
\begin{rem}\label{bari_compare_our}
Recently in {\rm \cite{BFpre}}, 
Barilari and Flynn proved 
the Strichartz estimates such as
\[
\|e^{it\mathcal{L}}u_0\|_{L^{r_2}_{\eta}L^q_tL^{r_1}_z}
\le C\|u_0\|_{H^{\sigma}}
\]
for
$\sigma
=\frac{N}{2}-\frac{2d}{r_1}-\frac{2p}{r_2}-\frac{2}{q}$
under the condition
\begin{equation}\label{admissi_bari}
r_1\le \min\{r_2,q\},\ \ r_2\ge 2+\frac{4}{p-1},\ \ 
\frac{2}{q}\le \frac{N}{2}-\frac{2d}{r_1}-\frac{2p}{r_2}, 
\end{equation}
where $z$ and $\eta$ denote the horizonal variable and vertical (center) variable of $g=(z,\eta)\in \htype$ 
respectively. 
Because
\[
\begin{split}
\|e^{it\mathcal{L}}u_0\|_{L^{q}_tL^{r}_g}
&\lesssim \|\mathcal{L}^{\frac{1}{2}\max\left\{N\left(\frac{1}{q}-\frac{1}{r}\right),0\right\}}e^{it\mathcal{L}}u_0\|_{L^{q}_tL^{\min\{r,q\}}_g}\\
&\lesssim  \|\mathcal{L}^{\frac{1}{2}\max\left\{N\left(\frac{1}{q}-\frac{1}{r}\right),0\right\}}e^{it\mathcal{L}}u_0\|_{L_{\eta}^{\min\{r,q\}}L^{q}_tL^{\min\{r,q\}}_z}
\end{split}
\]
holds by the Sobolev inequality and Minkowski's integral inequality, 
our Strichartz estimate {\rm (\ref{stri_hom_est})}
can be obtained by using the Strichartz estimate in {\rm \cite{BFpre}} 
with $r_1=r_2=\min\{r,q\}$. 
When $r_1=r_2=\min\{r,q\}$, the condition {\rm (\ref{admissi_bari})} is equivalent to
\[
\min\{q,r\}\ge 2+\frac{4}{p-1},\ \ \frac{2}{q}\le N\left(\frac{1}{2}-\frac{1}{\min\{q,r\}}\right),
\]
and the relation
\[
\max\left\{N\left(\frac{1}{q}-\frac{1}{r}\right),0\right\}+N\left(\frac{1}{2}-\frac{1}{\min\{q,r\}}\right)
=N\left(\frac{1}{2}-\frac{1}{r}\right)
\]
holds for any $(q,r)\in [2,\infty]^2$. 
Therefore, the Strichartz estimate in {\rm \cite{BFpre}} covers 
our Strichartz estimate only under the condition
\begin{equation}\label{admissi_bari_qr}
\min\{q,r\}\ge 2+\frac{4}{p-1},\ \ \frac{2}{q}\le N\left(\frac{1}{2}-\frac{1}{\min\{q,r\}}\right). 
\end{equation}
The condition {\rm (\ref{admissi_bari_qr})}, which is equivalent to 
\[
\max\left\{\frac{1}{q},\frac{1}{r}\right\}\le \frac{1}{2}-\frac{1}{p+1}
\]
 is stronger than our condition {\rm (\ref{sad_cond})}.  
 Therefore, the result in {\rm \cite{BFpre}} does not completely 
 contain {\rm Theorem~\ref{H-Strichartz}}. 
\end{rem}

The content of the paper is as follows. In Section 2, we introduce the definition and properties of the H-type groups $\htype$. Moreover, we summarize Sobolev spaces and Besov spaces on $\htype$ and give the fractional Leibniz rule on $\htype$\ (Proposition \ref{FLR}). 
In Section 3, we give the proof of Strichartz estimates (Theorem~\ref{H-Strichartz}). 
In Section 4, we give the proof of well-posedness  (Thorems~\ref{th1} and ~\ref{th_cri}) 
by applying the Strichartz estimates. 

Throughout this paper, the letters $C$ and so on will be used to denote positive constants, which are independent of the
main variables involved and whose values may vary at every occurrence.
By writing $f\lesssim g$, we mean $f\leq Cg$ for some positive constant $C>0$. The notation $f\sim g$ will stand for $f\lesssim g$ and $g\lesssim f$.

\section{H-type groups and function spaces}\label{stLg}
\subsection{Definition and properties of H-type groups}
\noindent

H-type groups were first introduced by A. Kaplan \cite{Kap3}.  
We recall the definition of H-type groups
 (see \cite{Cow}, \cite{Hi2005}, \cite{Kap3}, \cite{Kap2}, \cite{Liu}, \cite{Stein}, and reference therein). 
 \begin{defn}\label{def_Htype_g}
 Let ${\mathcal G}$ be a two-step nilpotent Lie algebra endowed with an inner product $\disp\left<\cdot,\cdot\right>$ and we denote by ${\mathfrak z}$ its center. 
 Then ${\mathcal G}$ is said to be of H type if ${\mathcal G}$ satisfies the following two conditions:
\begin{itemize}
\item $[{\mathfrak z}^{\perp}, {\mathfrak z}^{\perp}]={\mathfrak z}$
\item For any $S\in{\mathfrak z}$, we define the mapping $J_S$ from ${\mathfrak z}^{\perp}$ 
to ${\mathfrak z}^{\perp}$ by $\disp\left<J_Su,w\right>=\left<S,[u,w]\right>\ (u,w\in{\mathfrak z}^{\perp})$. 
If $|S|=1$, $J_S$ is an orthogonal mapping.
\end{itemize} 
Let $G$ be a connected and simply connected Lie group. 
Then $G$ is said to be a group of H type if its Lie algebra ${\mathcal G}$ is of H type. 
Let ${\mathfrak z}^{\ast}$ be the dual of ${\mathfrak z}$. 
For a given $a (\neq 0) \in{\mathfrak z}^{\ast}$, 
a skew-symmetric mapping $B(a)$ on ${\mathfrak z}^{\perp}$ is defined by
\[
\disp B(a)\left(u,w\right):=a([u,w]),\ u,w\in{\mathfrak z}^{\perp}.
\]
We denote by  $z_a$ an element of ${\mathfrak z}$ determined by
\[
\disp B(a)\left(u,w\right)=a([u,w])=\disp\left<J_{z_a}u,w\right>,\ u,w\in{\mathfrak z}^{\perp}. 
\]
Since $B(a)$ is non-degenerate and a symplectic form, 
we can see that the dimension of ${\mathfrak z}^{\perp}=2d$. 
For a given $a(\neq 0)\in{\mathfrak z}^{\ast}$, 
we can choose an orthonormal basis of ${\mathfrak z}^{\perp}$
\[
\{E_1(a),E_2(a),\cdots,E_d(a),\bar{E}_1(a),\bar{E}_2(a),\cdots,\bar{E}_d(a)\}
\]
such that 
\[
B(a)E_i(a)=|z_a|J_{\frac{z_a}{|z_a|}}E_i(a)=\varepsilon_i|z_a|\bar{E}_i(a)
\]
and
\[
B(a)\bar{E}_i(a)=-\varepsilon_i|z_a|{E}_i(a),
\]
where $\varepsilon_i=\pm 1$. Set $p=\mathrm{dim}~{\mathfrak z}$. 
Then we can denote  the elements of ${\mathcal G}$ by
\[
(z,\eta)=(x,y,\eta)=\disp\sum_{i=1}^d(x_iE_i+y_i\bar{E}_i)+\disp\sum_{j=1}^{p}\eta_j\tilde{E}_j,
\] 
where $\{\tilde{E}_1,\cdots,\tilde{E}_p\}$ is an orthonormal basis such that $a(\tilde{E}_1)=|a|$, 
$a(\tilde{E}_j)=0$, $(j=2,3,\cdots,p)$. 
We identify the H-type Lie algebra ${\mathcal G}$ with the H-type Lie group $G$. 
Then the group law on the H-type group has the form    
\begin{align*}
(z,\eta)\circ(z^{\prime},\eta^{\prime})=\disp\left(z+z^{\prime}, \eta+\eta^{\prime}+\disp\frac{1}{2}[z,z^{\prime}]\right), 
\end{align*}
where $[z,z^{\prime}]_j=\disp\left<z,U^jz^{\prime}\right>$\ $(j=1,2,\cdots,p)$ 
and $U^j$ satisfies the following conditions:
\begin{enumerate}
\item[{\rm (a)}] $U^j$ is a $2d\times 2d$ skew-symmetric and orthogonal matrix, 
\item[{\rm (b)}] For any $i,j\in\{1,2,\cdots,p\}$, $i\neq j$, $U^iU^j+U^jU^i=0$.
\end{enumerate} 
In what follows, we denote by ${\mathbb H}^d_p(={\mathbb R}^{2d+p})$ H-type groups $\htype$ 
to emphasize the dimension $p$ of the center, instead of $G$.
\end{defn}
\begin{rem}\label{charaHtype}
\begin{enumerate}
\item[{\rm (1)}] H-type groups $\htype$ must satisfy $p+1\leq 2d$ $($see {\rm \cite{Kap2}}$)$. 
\item[{\rm (2)}] If the matrix $U^j$ is skew symmetric $($linearly independent$)$, then $\htype$ is called Carnot group.
\item[{\rm (3)}] \label{Hormander1}For any $\lambda >0$, the dilation $\delta_{\lambda}$ : ${\mathbb R}^{2d+p}\rightarrow{\mathbb R}^{2d+p}$ defined by 
\[
\delta_{\lambda}(z,\eta):=(\lambda z,\lambda^2\eta)
\]
for $z=(x,y)\in{\mathbb R}^{2d}$ and $\eta\in{\mathbb R}^p$, 
is an automorphism of H-type groups $\htype$.
\item[{\rm (4)}] If $p=1$, then $\mathbb{H}^d_1$ is called the Heisenberg group.
\end{enumerate}
\end{rem}

By Definition~\ref{def_Htype_g}, the unit element of H-type groups is $\bvec{e}=(0,0)$ and 
the inverse element is $(-z,-\eta)$. 
For $j=1,\cdots, d$ and $i=1,\cdots,p$, 
the left-invariant vector fields are given by 
\begin{align*}
X_j &:= \disp\frac{\partial}{\partial x_j}+\frac{1}{2}\disp\sum_{k=1}^p\left(\disp\sum_{l=1}^{2d}z_lU_{l,j}^k\right)\disp\frac{\partial}{\partial \eta_k},\ 
Y_j := \disp\frac{\partial}{\partial y_j}+\frac{1}{2}\disp\sum_{k=1}^p\left(\disp\sum_{l=1}^{2d}z_lU_{l,j+d}^k\right)\disp\frac{\partial}{\partial \eta_k},\ 
S_i :=\disp\frac{\partial}{\partial \eta_i},
\end{align*} 
where, $z_l=x_l$, $z_{l+d}=y_l$ $(l=1,2,\cdots,d)$ and 
$U_{i,j}^k$, $U_{i,j+d}^k$ are the $(i,j)$ and $(i,j+d)$ components of the matrix $U^k$, respectively. 
Let 
\[
{\mathcal B}_0=(X_1,\cdots,X_{d},Y_1,\cdots,Y_{d}),\ \ {\mathcal F}_0=(S_1,\cdots,S_p)
\]
 be an orthonormal basis of 
${\mathfrak z}^{\perp}$ and   an orthonomal basis of ${\mathfrak z}$, respectively. 
By using these basis, we identify ${\mathfrak z}^{\perp}$ with ${\mathbb R}^{2d}$ and 
${\mathfrak z}$ with ${\mathbb R}^p$, respectively. 
Then H${\rm\ddot{o}}$rmander condition
\begin{align}
\mathrm{rank}(\mathrm{Lie}\{X_1,\cdots,X_{d},Y_1,\cdots,Y_{d}\}(g))=2d+p \label{Hormander2}
\end{align}
holds for any $g\in{\mathbb R}^{2d+p}$, 
that is, 
the iterated commutators of $X_1,\cdots,X_{d}$, 
$Y_1,\cdots,Y_d$ span the Lie algebla ${\mathcal G}$ of ${\htype}$. 
Hence by Remark \ref{charaHtype} {\rm (3)} and \eqref{Hormander2}, 
H-type groups $\htype$ are 2-step stratified Lie groups 
(regarding the details of stratified Lie groups, we refer to
 \cite{Ugu}, \cite{Fis}, \cite{Folland2}, 
 \cite{Var2} and reference therein). 
 
The sublaplacian of $\htype$ is denoted by 
\[
{\mathcal L}:=-\sum_{i=1}^{d}(X_i^2+Y_i^2).
\] 
This essentially self-adjoint positive operator does not depend on the choice of ${\mathcal B}_0$. 
Thanks to H${\rm\ddot{o}}$rmander's result, the sublaplacian ${\mathcal L}$ is subelliptic.  
This does not depend on the choice of ${\mathcal B}_0$ and ${\mathcal F}_0$ (see \cite{Cor}).  
Furthermore, thanks to H${\rm\ddot{o}}$rmander's result, the Carnot-Carath${\rm \acute{e}}$odory distance 
$\rho_{{\mathcal B}_0}(g,g^{\prime})$ can also be defined (see \cite{Ugu} and \cite{Var2}  for details). 
We denote by $\rho(g)$ the distance from the origin, {\it i.e}. $\rho(g)=\rho_{{\mathcal B}_0}(\bvec{e},g)$.  
The homogeneous of degree of  $\rho$  is one, that is, 
\begin{align*}
\rho(\delta_{\lambda}(g))=\lambda \rho(g),\ g\in {\htype}
\end{align*}
for any $\lambda>0$ (see \cite{Ugu}, Proposition 5.2.6).
It also holds that
\begin{align*}
 \rho({g^{\prime}}^{-1}\cdot g)\leq \rho(g)+\rho({g^{\prime}}).
\end{align*}

Let
\[
N:=\mathrm{dim}~{\mathfrak z}^{\perp}+2\mathrm{dim}~{\mathfrak z}=2d+2p
\]
denotes the homogeneous dimension of $\htype$. 
H-type groups $\htype$ are locally compact Hausdorff spaces and Haar measure of  $\htype$ is 
the Lebesgue measure 
\[
dg=dx_1\cdots d{x_d}dy_1\cdots dy_{d}d\eta_1\cdots d\eta_p.
\] 
We can see that
\[
\disp\int_{{\htype}}f(\delta_{\lambda}(g))dg=\lambda^{-N}\disp\int_{\htype}f(g)dg.
\]
The convolution $f*h$ of $f$ with $h$ on $\htype$ is defined by
\[
(f*h)(g):=\disp\int_{\htype}f(g^{\prime})h({g^{\prime}}^{-1}\cdot g)dg^{\prime}
=\disp\int_{\htype}f(g\cdot {g^{\prime}}^{-1})h(g^{\prime})dg^{\prime}.
\] 
The convolution $*$ is non-commutative. 
The relationship between the left-invariant vector fields $X_i$ and 
the convolution $*$ is $X_i(f*h)(g)=(f* X_ih)(g)$.
For $1\le q\le \infty$, we set 
\[
L^q(\htype):=\disp\left\{f\ |\ \|f\|_{L^q}<\infty\right\}\\
\]
with the norm $\|\cdot\|_{L^q}$ defined by 
\[
\|f\|_{L^q}:=
\begin{cases}
\ \left(\disp\int_{\htype}|f(g)|^qdg\right)^{\frac{1}{q}} &{\rm if}\ 1\le q<\infty,\\ 
\ \displaystyle \esssup\displaylimits_{g\in \htype}f(g) &{\rm if}\ q=\infty. 
\end{cases}
\]

\subsection{Besov and Sobolev spaces on H-type groups}\label{bes_sob_def_sec}
\noindent

At first, we recall the definitions of 
the Besov spaces $B^s_{r,q}(\htype)$ and $\dot{B}^s_{r,q}(\htype)$. 
By the spectral theorem, the sublaplacian ${\mathcal L}$ on H-type groups $\htype$ 
satisfies a spectral resolution
\[
{\mathcal L}=\disp\int_0^{\infty}\lambda dE_{\lambda},
\]
where $dE_{\lambda}$ is the projection measure. 
If $\Theta$ is a bounded Borel measure function on ${\mathbb R}_+$, 
then the operator
\[
\Theta({\mathcal L})=\disp\int_0^{\infty}\Theta({\lambda})dE_{\lambda}
\]
is bounded on $L^2(\htype)$. 
Furthermore by the Schwartz kernel theorem, 
there exists a tempered distribution kernel $K_{\Theta({{\mathcal L}})}$ on $\htype$ 
such that 
\[
\Theta({\mathcal L})f=f*K_{\Theta({{\mathcal L}})}
\]
for any $f\in{\mathcal S}(\htype)$, 
where ${\mathcal S}$ denotes the Schwartz class. 
It is known that if $\Theta\in{\mathcal S}({\mathbb R}_+)$, 
then the distribution kernel $K_{\Theta({{\mathcal L}})}$ of 
the operator $\Theta({\mathcal L})$ belongs to ${\mathcal S}(\htype)$
(see, \cite{Fur}, \cite{Hu}, and \cite{Hul2}).
 
 Let $\varphi, \varphi_0\in C^{\infty}({\mathbb R}_+)$ such that 
$\mathrm{supp}~\varphi\subset [0,4]$, $|\varphi(\lambda)|\geq c>0$ for $\lambda\in[0, 2^{3/2}]$
and 
$\mathrm{supp}~\varphi_0\subset[1/4, 4]$, $|\varphi_0(\lambda)|\geq c>0$ for $\lambda\in[2^{-3/2}, 2^{3/2}]$. 
Set $\varphi_j(\lambda)=varphi_0(2^{-2j}\lambda)$ for $j\in \Z$.
We define the Besov spaces as follows (we refer to \cite{Fue}, \cite{Fur}, \cite{Hu2}, \cite{Hu}, and references therein). 
\begin{defn}\label{HomoBes}
Let $s\in{\mathbb R}$, $1\leq r<\infty$ and $1\leq q\leq \infty$. 
\begin{enumerate}
\item[{\rm (i)}] 
The inhomogeneous Besov space ${B}^s_{r,q}(\htype)$ is defined as the set of all $f\in {\mathcal S}^{\prime}(\htype)$ for which 
\begin{align}
\|f\|_{{B}^s_{r,q}}:=\disp\left(\left\|\varphi({{\mathcal L}})f\right\|_{L^r}^q
+\disp\sum_{j=1}^{\infty}2^{jsq}\left\|\varphi_j({{\mathcal L}})f\right\|_{L^r}^q\right)^{\frac{1}{q}}<\infty\label{241262}
\end{align}
with the usual modification for $q=\infty$.
\item[{\rm (ii)}] 
The homogeneous Besov space $\dot{B}^s_{r,q}(\htype)$ is defined as the set of all $f\in {\mathcal S}^{\prime}(\htype)/{\mathcal P}$ for which 
\begin{align}
\|f\|_{\dot{B}^s_{r,q}}:=\disp\left(\disp\sum_{j\in{\mathbb Z}}2^{jsq}\left\|\varphi_j({{\mathcal L}})f\right\|_{L^r}^q\right)^{\frac{1}{q}}<\infty\label{241262}
\end{align}
with the usual modification for $q=\infty$, where ${\mathcal P}$ denotes the space of all polynomials on $\htype$.
\end{enumerate}
\end{defn}
\noindent

\begin{rem}\label{rem_def_space}
\begin{enumerate}
\item[{\rm (i)}] By the general theory developed in {\rm \cite{Fue}, \cite{Hu2}}, and {\rm \cite{Hu}}, 
it is known that the definitions of the these spaces are independent of the choice of $\varphi$ and $\varphi_0$, as long as $\varphi_0$ and $\varphi$ satisfy all the conditions as above. 
\item[{\rm (ii)}] 
Suppose $\varphi$, $\varphi_0\in C^{\infty}({\mathbb R}_{\geq 0})$ such that $\mathrm{supp}$ $\varphi$  and $\mathrm{supp}$ $\varphi_0$ are compact, $0\not\in \mathrm{supp}$ $\varphi_0$, and 
\begin{align}\label{assump2}
\varphi (\lambda)+\disp\sum_{j=1}^{\infty}\varphi_j(\lambda)=1\ \ (\lambda\in{\mathbb R}_{\geq 0}),
\end{align}
then for all $f\in {\mathcal S}'(\htype)$,
\[
f=\varphi ({\mathcal L})f+\disp\sum_{j=1}^{\infty}\varphi_j({\mathcal L})f\ {\rm in}\ S'(\htype)
\]
holds $($see {\rm \cite{Hu2}}$)$.

Also note that if $\varphi_0\in C^{\infty}({\mathbb R}_+)$ with compact support, vanishing identically near the origin, and satisfying
\begin{align}\label{assump1}
\disp\sum_{j\in{\mathbb Z}}\varphi_j(\lambda)=1\ \ (\lambda\in{\mathbb R}_+), 
\end{align}
then for all $f\in {\mathcal S}^{\prime}(\htype)/{\mathcal P}$,
\[
f=\disp\sum_{j\in\mathbb{Z}}\varphi_j({{\mathcal L}})f\ {\rm in}\ {\mathcal S}^{\prime}(\htype)/{\mathcal P}
\] 
holds $($see {\rm \cite{Hu}}$)$. 

Therefore, in this paper, we also assume  \eqref{assump2} and \eqref{assump1} with 
the conditions of $\varphi$ and $\varphi_0$ as above. 
\item[{\rm (iii)}] 
By {\rm Corollary~3.16} in {\rm \cite{Hu2}}, for $s\in{\mathbb R}$, $1\leq r,q\leq\infty$, a nonnegative integer $m$ such that $m>s$, and $f\in{\mathcal S}'(\htype)$, it holds that  
\[
\|f\|_{B^s_{r,q}}\sim \|e^{-{\mathcal L}}f\|_{L^r}+\left(\disp\int_0^1\xi^{-sq/2}\|(\xi{\mathcal L})^{m/2}e^{-\xi{\mathcal L}}f\|_{L^r}^q\dfrac{d\xi}{\xi}\right)^{1/q}.
\]
Therefore, by {\rm Theorem~5.1.(iii)} in {\rm \cite{Bruno2}}, for $1\leq r\leq \infty$, we have the embedding  
\begin{align}
B^{{N}/{r}}_{r,1}(\htype)\hookrightarrow L^{\infty}(\htype).
\end{align}
Note that we consider only the case that the group $G$ in {\rm \cite{Bruno2}} is unimodular.
\item[{\rm (iv)}] If $s<N/r$, then 
$\dot{B}^s_{r,q}(\htype)$ can be defined as the set of tempered distributions $f\in S'(\htype)$ such that \eqref{241262} holds  $($see {\rm \cite{Bah}}, {\rm \cite{Saw}} and reference therein$)$.
\end{enumerate}
\end{rem}

Next, we recall the definition and basic properties of Sobolev spaces on $\htype$. 
We adopt the definition of the Sobolev spaces in \cite{Fis} to H-type groups $\htype$
(see also \cite{Folland2} and \cite{Sa1979}).
At first, we recall the definition of fractional powers of the sublaplacian ${\mathcal L}$.
\begin{defn}[\cite{Fis}, \cite{Folland2}, \cite{Sa1979}]
Assume that $1<r<\infty$, $s>0$ and $k=[s]+1$. 
Then the operator ${\mathcal L}_r^{s}$ is defined by 
\begin{align*}
{\mathcal L}_r^{s}f =\disp\lim_{\varepsilon\rightarrow 0}\disp\frac{1}{\Gamma(k-s)}\disp\int_{\varepsilon}^{\infty}\nu^{k-s-1}{\mathcal L}^ke^{\nu{\mathcal L}}f~d \nu
\end{align*}
on the domain of all $f\in L^r(\htype)$ such that the indicated limit exists in $L^r(\htype)$. 
The operator ${\mathcal L}_r^{-s}$ is defined by 
\begin{align*}
{\mathcal L}_r^{-s}f=\disp\lim_{\eta\rightarrow \infty}\disp\frac{1}{\Gamma(s)}\disp\int_0^{\eta}\nu^{s-1}e^{\nu{\mathcal L}}f~d\nu
\end{align*}
on the domain of all $f\in L^r(\htype)$ such that the indicated limit exists in $L^r(\htype)$. 
The operator $(\mathrm{Id}+{\mathcal L}_r)^{s}$ is defined by
\begin{align*}
(\mathrm{Id}+{\mathcal L}_r)^{s}f = \disp\lim_{\varepsilon\rightarrow 0}\disp\frac{1}{\Gamma(k-s)}\disp\int_{\varepsilon}^{\infty}\nu^{k-s-1}(\mathrm{Id}+{\mathcal L})^ke^{-\nu}e^{\nu{\mathcal L}}f~d\nu
\end{align*}
on the domain of all $f\in L^r(\htype)$ such that the indicated limit exists in $L^r(\htype)$.
Also, we define the operator $(\mathrm{Id}+{\mathcal L}_r)^{-s}$ by 
\begin{align*}
(\mathrm{Id}+{\mathcal L}_r)^{-s}f=\disp\frac{1}{\Gamma(s)}\disp\int_0^{\infty}\nu^{s-1}
e^{-\nu}e^{\nu{\mathcal L}}f~d\nu.
\end{align*}
The operator $(\mathrm{Id}+{\mathcal L}_r)^{-s}$ is a bounded operator on $L^r$. 
\end{defn}
\noindent
\begin{prop}[\cite{Fis}, \cite{Folland2}]\label{3_15_1}
Let $1<r<\infty$ and ${\mathcal M}_r$ denotes either ${\mathcal L}_r$ or $\mathrm{Id}+{\mathcal L}_r$. 
\begin{enumerate}
\item[\rm{(i)}] ${\mathcal M}_r^s$ is a closed operator on $L^r(\htype)$ for all  $s\in{\mathbb R}$ and 
injective with $({\mathcal M}^s_r)^{-1}={\mathcal M}^{-s}_r$.
\item[\rm{(ii)}] If $f\in \mathrm{Dom}({\mathcal M}^{\beta}_r)\cap\mathrm{Dom}({\mathcal M}^{\alpha+\beta}_r)$, then ${\mathcal M}^{\beta}_rf\in\mathrm{Dom}({\mathcal M}^{\alpha}_r)$ and ${\mathcal M}^{\alpha}_r{\mathcal M}^{\beta}_rf={\mathcal M}^{\alpha+\beta}_rf$. ${\mathcal M}^{\alpha+\beta}_r$ becomes the smallest closed extension of ${\mathcal M}^{\alpha}_r{\mathcal M}^{\beta}_r$. 
\item[\rm{(iii)}] When $s>0$, if $f\in \mathrm{Dom}({\mathcal M}^{s}_r)\cap L^q(\htype)$, 
then $f\in \mathrm{Dom}({\mathcal M}^{s}_q)$ if and only if ${\mathcal M}^{s}_rf\in L^q(\htype)$, 
in which case ${\mathcal M}^{s}_r={\mathcal M}^{s}_q$. 
\item[\rm{(iv)}] If $s>0$, 
then $\mathrm{Dom}({\mathcal L}_r^s)=\mathrm{Dom}((\mathrm{Id}+{\mathcal L}_r)^s)$ .
\end{enumerate}
\end{prop}
\noindent
By Proposition \ref{3_15_1} (iii), 
${\mathcal L}_r^s$ (resp. $(\mathrm{Id}+{\mathcal L}_r)^s$) agrees with 
${\mathcal L}_q^s$ (resp. $(\mathrm{Id}+{\mathcal L}_q)^s$) on their common domains 
for $s\in{\mathbb R}$ and $1<q,r<\infty$. 
So we omit the subscripts on these operators except when we wish to specify the domains.

We define the definition of Sobolev spaces $W^{s,r}(\htype)$ and $\dot{W}^{s,r}(\htype)$
as follows. 
\begin{defn}[\cite{Fis}]\label{3_1_1}Let $s\in{\mathbb R}$ and $1<r<\infty$. 
\begin{enumerate}
\item[{\rm (i)}]We denote by $W^{s,r}(\htype)$ the space of tempered distributions obtained by the completion of 
the Schwartz class ${\mathcal S}(\htype)$ with respect to the Sobolev norm 
\[
\|f\|_{W^{s,r}}:=\|(\mathrm{Id}+{\mathcal L})^{\frac{s}{2}}f\|_{L^r}.
\]
\item[{\rm (ii)}]We denote by $\dot{W}^{s,r}(\htype)$ the space of tempered distributions obtained by the completion of 
${\mathcal S}(\htype)\cap\mathrm{Dom}~({\mathcal L}^{\frac{s}{2}})$ with respect to the norm 
\[
\|f\|_{\dot{W}^{s,r}}:=\|{\mathcal L}^{\frac{s}{2}}f\|_{L^r}
\]
\end{enumerate}
\end{defn}
\noindent
The Sobolev spaces $W^{s,r}(\htype)$ and $\dot{W^{s,r}}(\htype)$ have the following basic properties.
\begin{prop}[\cite{Fis}]\label{3_31_1}
\begin{enumerate}
\item[{\rm (i)}] Let $s\in {\mathbb R}$ and $1<r<\infty$. 
Then $W^{s,r}(\htype)$ and $\dot{W^{s,r}}(\htype)$ are 
Banach space satisfying
\[
{\mathcal S}(\htype)\subsetneq W^{s,r}(\htype)\subset{\mathcal S}^{\prime}(\htype)
\]
and 
\[
({\mathcal S}(\htype)\cap\mathrm{Dom}~({\mathcal L}_r^{\frac{s}{2}}))\subsetneq \dot{W}^{s,r}(\htype)\subsetneq{\mathcal S}^{\prime}(\htype),
\]
 respectively. 
\item[{\rm (ii)}] If $s=0$ and $1<r<\infty$, then $\dot{W}^{0,r}(\htype)=W^{0,r}(\htype)=L^r(\htype)$ with 
$\|\cdot\|_{\dot{L}^r_0}=\|\cdot\|_{L^r_0}=\|\cdot\|_{L^r}$.
\item[{\rm (iii)}] If $s>0$ and $1<r<\infty$, then we have 
\[
W^{s,r}(\htype)=\dot{W}^{s,r}(\htype)\cap L^r(\htype)
\]
and 
\[
\|\cdot\|_{W^{s,r}}\sim \|\cdot\|_{L^r}+\|\cdot\|_{\dot{W}^{s,r}}.
\]
 \end{enumerate}
\end{prop}
\begin{rem}
We denote the space $\dot{W}^{s,2}(\htype)$ by $\dot{H}^s(\htype)$ and 
the space ${W}^{s,2}(\htype)$ by ${H}^s(\htype)$, respectively.
\end{rem}
%
%
%

\begin{prop}\label{embeddings}If $1< r\leq \infty$ and $s\in{\mathbb R}$, then we have the estimates
\begin{equation}\label{besov_sobo_embed}
\begin{split}
&\|f\|_{\dot{B}^s_{r,2}}\le C\|f\|_{\dot{W}^{s,r}},\ \ 
1<r\le 2,\\
&\|f\|_{\dot{W}^{s,r}}\le C\|f\|_{\dot{B}^s_{r,2}},\ \ 2\le r<\infty. 
\end{split}
\end{equation}
\end{prop}
\begin{proof}
For $1<r<2$, one can find it in Theorem 3.3 (4) in \cite{Car}.
For $r\geq 2$, it is immediately clear from the relation
\begin{equation}\label{sob_tri_rel}
\|f\|_{\dot{F}^{s}_{r,2}}
\sim \|f\|_{\dot{W}^{s,r}}
\end{equation}
and Minkowski's integral inequality, 
where $\dot{F}^{s}_{r,2} $ denotes the homogeneous Triebel-Lizorkin space defined by the norm
\[
\|f\|_{\dot{F}^s_{r,2}}=
\disp\left\|\left(\disp\sum_{j\in{\mathbb Z}}2^{2js}|\varphi_j({{\mathcal L}})f|^2\right)^{\frac{1}{2}}\right\|_{L^r}. 
\] 
The proof of $(${\rm \ref{sob_tri_rel}}$)$ is given in {\rm \cite{Hir}} $($see, {\rm Lemma~6.2} in {\rm \cite{Hir}}$)$. 
\end{proof}
\begin{rem}
The inhomogeneous version of
{\rm (\ref{besov_sobo_embed})} is obtained by Saka\ $(${\rm Theorem\ 20} in {\rm \cite{Sa1979}}$)$. 
\end{rem}
\begin{rem}
By {\rm Proposition~\ref{embeddings}}, we have
\[
 \|f\|_{\dot{B}^s_{r,2}}\sim \|f\|_{\dot{H}^{s}}= \|{\mathcal L}^{\frac{s}{2}}f\|_{L^2}.
\]
\end{rem}
%
%
Finally in this subsection, we give the fractional Leibniz rule on $\htype$.
\begin{lemm}[{\rm Lemma 5.1 in \cite{Hir}}]\label{frac_chain_ref}
Assume that $F\in C^l(\C,\C)$, $l\in{\mathbb N}$, and $\alpha\geq l$ with
\[
F(0)=0,\ |F^{(j)}(z)|\leq C|z|^{\alpha-j}\ (z\in \C)
\]
for $j=1,2,\cdots, l$. 
Let $0\leq s\leq l$ and $1<p,q<\infty$, $1<r\leq \infty$ satisfy $\frac{1}{p}=\frac{1}{q}+\frac{\alpha-1}{r}$. Then, there exists $C>0$ such that
\[
\|{\mathcal L}^{\frac{s}{2}}F(u)\|_{L^p} \leq C\|u\|_{L^r}^{\alpha-1}\|{\mathcal L}^{\frac{s}{2}}u\|_{L^q}.
\]
holds for any $u\in L^r(\G)\cap\dot{W}^{s,q}(\G)$. 
\end{lemm}
{\rm Lemma~\ref{frac_chain_ref}} implies the following (see, Corollary~3.5 in \cite{Dinh17}). 
\begin{prop}\label{FLR}
Let $F(z)=|z|^{\alpha-1}z$ with $\alpha>1$, $s\geq 0$ and $1<p,q<\infty$, $1<r\leq\infty$ 
satisfy $\frac{1}{p}=\frac{1}{q}+\frac{\alpha-1}{r}$. 
Assume $\alpha \ge \lceil s\rceil$ if $\alpha$ is not an odd integer. 
Then, there exists $C>0$ such that
\begin{align}\label{FLR_ineq}
\|F(u)\|_{\dot{W}^{s,p}}\leq C\|u\|_{L^r}^{\alpha-1}\|u\|_{\dot{W}^{s,q}}.
\end{align}
holds for any $u\in L^r(\htype)\cap \dot{W}^{s,q}(\htype)$. 
Additionally assume 
$\alpha \ge \lceil s\rceil +1$ if $\alpha$ is not an odd integer. 
Then, there exists $C>0$ such that
\begin{equation}\label{FLR_diff_ineq}
\begin{split}
\|F(u)-F(v)\|_{\dot{W}^{s,p}}
&\leq C(\|u\|_{L^r}^{\alpha-1}+\|v\|_{L^{r}}^{\alpha-1})
\|u-v\|_{\dot{W}^{s,q}}\\
&\ \ \ \ +C(\|u\|_{L^r}^{\alpha-2}+\|v\|_{L^r}^{\alpha-2})(\|u\|_{\dot{W}^{s,q}}+\|v\|_{\dot{W}^{s,q}})
\|u-v\|_{L^r}
\end{split}
\end{equation}
holds for any $u,v\in L^r(\htype)\cap \dot{W}^{s,q}(\htype)$. 
\end{prop}
\begin{rem}\label{IFLR}
By H\"older inequality and {\rm Proposition \ref{3_31_1}\ (iii)}, 
the same estimates as in {\rm Proposition~\ref{FLR}} hold even if $\dot{W}^{s,p}$ and $\dot{W}^{s,q}$-norms are replaced by ${W}^{s,p}$ and ${W}^{s,q}$-norms respectively.
\end{rem}

\subsection{Spherical Fourier transform}\label{SFT_subsec}
\noindent 

A function $f$ on $\htype$ is said to be radial if the value of $f(z,s)$ depends only on $|z|$ and $s$.
We denote by ${\mathcal S}_{\mathrm{rad}}(\htype)$ and by $L^p_{\mathrm{rad}}(\htype)$, $1\leq p\leq\infty$, 
the space of radial functions in ${\mathcal S}(\htype)$ and in $L^p(\htype)$, respectively. Note that the set of $L^1_{\mathrm{rad}}(\htype)$ endowed with the convolution product $*$ is a commutative.

For $f\in L^1_{\mathrm{rad}}(\htype)$, we define the spherical Fourier transform
\begin{equation}\label{sp_ft_def}
\hat{f}(\lambda,m)=
\begin{pmatrix}
m+d-1\\
m
\end{pmatrix}^{-1}\disp\int_{\mathbb{R}^{2d+p}}e^{i\lambda s} f(z,s)L^{(d-1)}_m\left(\dfrac{|\lambda|}{2}|z|^2\right)dzds
\end{equation}
for $m\in{\mathbb N}$ and $\lambda\in{\mathbb R}^p$, where $L_m^{(d-1)}(\tau)$ is the Laguerre functions (see \cite{Hi2005} and \cite{Liu}). Note that $\widehat{f_1*f_2}=\hat{f_1}\hat{f_2}$ holds for $f_1,f_2\in L^1_{\mathrm{rad}}(\htype)$. The following proposition is Plancherel theorem on H-type groups.
\begin{prop}[\cite{Hi2005}, \cite{Liu}]\label{inverse_formula}
For all $f\in {\mathcal S}_{\mathrm{rad}}(\htype)$ satisfying
\[
\disp\sum_{m\in{\mathbb N}}
\begin{pmatrix}
m+d-1\\
m
\end{pmatrix}\disp\int_{{\mathbb R}^p}\left|\hat{f}(\lambda,m)\right||\lambda|^dd\lambda<\infty,
\]
we have
\[
f(z,s)=\left(\dfrac{1}{2\pi}\right)^{d+p}\sum_{m\in{\mathbb N}}\disp\int_{{\mathbb R}^p}e^{-i\lambda s}\hat{f}(\lambda, m)
L^{(d-1)}_m\left(\dfrac{|\lambda|}{2}|z|^2\right)|\lambda|^dd\lambda,
\]
where the sum being convergent in $L^{\infty}$ norm.
\end{prop}
Furthermore, if $f\in{\mathcal S}_{\mathrm{rad}}(\htype)$, the function ${\mathcal L}f$ are also in  ${\mathcal S}_{\mathrm{rad}}(\htype)$. So its spherical transform is given by
\begin{align}\label{EV1}
\widehat{{\mathcal L}f}(\lambda,m)=(2m+d)|\lambda|\hat{f}(\lambda,m).
\end{align}
Let $A=i{\mathcal L}$. Since ${\mathcal L}$ is a self-adjoint operator in $L^2(\htype)$, $iA$ is also a self-adjoint operator in $L^2(\htype)$. Then by Stone's theorem, the family of the multiplier operators $e^{it{\mathcal L}}$ for $t\in{\mathbb R}$ is a group of 
unitary operators (for example, see \cite{Pazy}). Thus it holds that $(e^{it{\mathcal L}})^{*}=e^{-it{\mathcal L}}$ and $\|e^{it{\mathcal L}}f\|_{L^2}=\|f\|_{L^2}$ for all $f\in L^2(\htype)$. Furthermore, if $f\in L_{\rm rad}^2(\htype)$, by \eqref{EV1}, we have
\begin{align}\label{EV2}
\widehat{e^{it{\mathcal L}}f}(\lambda,m)=e^{it (2m+d)|\lambda|}\hat{f}(\lambda,m).
\end{align}
For any $j\in{\mathbb Z}$, we define by $\Phi_j$ the kernel of the operator $\varphi_j({\mathcal L})$ which has already appeared in subsection~\ref{bes_sob_def_sec}
(also be careful of Remark \ref{rem_def_space}). Especially, $\Phi_0$ is the kernel of the operator $\varphi_0({\mathcal L})$. Furthermore, by Proposition \ref{inverse_formula}, we have the homogeneous property
\begin{align}\label{phi_j_oomoji_def}
\Phi_j(z,s)=2^{Nj}\Phi_0(2^jz,2^{2j}s).
\end{align}
Since $\Phi_j\in{\mathcal S}_{\mathrm{rad}}(\htype)$, we also have
\begin{align}\label{EV3}
\widehat{\Phi_j}(\lambda,m)=\varphi_0(2^{-2j}(2m+d)|\lambda|).
\end{align}
We set
\[
\widetilde{\Phi}_j=\Phi_{j-1}+\Phi_j+\Phi_{j+1},\ j\in{\mathbb Z}.
\]
Then it holds that
\[
\widehat{\Phi_j}(\lambda,m)=\widehat{\Phi_j}(\lambda,m)\widehat{\widetilde{\Phi}_j}(\lambda,m).
\]
Hence we can find that 
\begin{align}
\Phi_j=\Phi_j*\widetilde{\Phi}_j,\ j\in{\mathbb Z}.
\end{align}
For convenience, we  also define the operator $\Delta_j$ by
\[
\Delta_jf:=f*\Phi_j.
\]
%
\section{Proof of the Strichartz estimate}
\noindent

In this section, we prove the Strichartz estimate (Theorem~\ref{H-Strichartz}). 
The proof is based on the duality argument (see, \cite{CW90} and \cite{GV}) 
with dispersive estimate. 
First, we introduce the time decay estimate which is given by Hierro. 
\begin{prop}[{\rm Lemma~5.1 in \cite{Hi2005}}]\label{disp_hierro}
Let $d,p\in \N$ with $p\ge 2$. 
Let $\varphi_0$ is given in {\rm Subsection~\ref{bes_sob_def_sec}} and 
$\Phi_0$ is the kernel of $\varphi_0(\mathcal{L})$ on $\htype$ as in {\rm (\ref{EV3})}. Then, 
there exists $C>0$ such that 
\[
\|e^{it\mathcal{L}}\Phi_0\|_{L^{\infty}}\le C\min\{1,|t|^{-\frac{p-1}{2}}\}\ (\lesssim\ (1+|t|))^{-\frac{p-1}{2}})
\]
holds for any $t\in \R\setminus \{0\}$. 
\end{prop}
For $t\in \R \setminus\{0\}$, we define the operators $\Delta_{\rm L}(t)$ and $\Delta_{\rm H}(t)$ as
\[
\Delta_{\rm L}(t)f=
\sum_{j<\log_2(1/\sqrt{|t|})}f*\Phi_j,\ \ \ 
\Delta_{\rm H}(t)f=
\sum_{j\ge \log_2(1/\sqrt{|t|})}f*\Phi_j. 
\]

By using Proposition~\ref{disp_hierro}, we obtain the following. 
\begin{prop}\label{disp_linf_est}
Let $d,p\in \N$ with $p\ge 2$. 
There exists $C>0$ such that
\[
\|e^{it\mathcal{L}}u_0\|_{L^{\infty}}\le 
C\left\{\|\Delta_{\rm L}(t)u_0\|_{\dot{B}^{N}_{1,1}}+|t|^{-\frac{p-1}{2}}\|\Delta_{\rm H}(t)u_0\|_{\dot{B}^{N-p+1}_{1,1}}\right\}
\]
holds
for any $t\in \R\setminus \{0\}$ and $u_0\in \dot{B}^{N-p+1}_{1,1}(\htype)$. 
\end{prop}
\begin{rem}\label{disp_rem_hierro}
In {\rm \cite{Hi2005}}, 
the similar estimate 
\[
\|e^{it\mathcal{L}}u_0\|_{L^{\infty}}\le C(1+|t|)^{-\frac{p-1}{2}}\|u_0\|_{B^{N-\frac{p-1}{2}}_{1,1}}
\]
is claimed. 
But the proof of this estimate cannot be found in {\rm \cite{Hi2005}}. 
It seems that the regularity loss $N-\frac{p-1}{2}$ 
is not suitable from viewpoint of scaling argument. 
We prove {\rm Proposition~\ref{disp_linf_est}} 
based on scaling argument. 
\end{rem}
\begin{rem}
We note that if $u_0\in \dot{B}^{N-p+1}_{1,1}(\htype)$, then $\Delta_{\rm L}(t)u_0\in \dot{B}^{N}_{1,1}(\htype)$ since
\begin{equation}\label{low_freq_dec_est1}
\|\Delta_{\rm L}(t)u_0\|_{\dot{B}^{N}_{1,1}}\le |t|^{-\frac{p-1}{2}}\|\Delta_{\rm L}(t)u_0\|_{\dot{B}^{N-p+1}_{1,1}}
\end{equation}
holds. Because we can see that
\[
\|\Delta_{\rm L}(t)u_0\|_{\dot{B}^{N-p+1}_{1,1}}+ \|\Delta_{\rm H}(t)u_0\|_{\dot{B}^{N-p+1}_{1,1}}
\sim \|u_0\|_{\dot{B}^{N-p+1}_{1,1}}
\]
by the definition of the Besov norm $($see, {\rm (\ref{241262})}$)$, 
we also obtain
\[
\|e^{it\mathcal{L}}u_0\|_{L^{\infty}}\le 
C|t|^{-\frac{p-1}{2}}\|u_0\|_{\dot{B}^{N-p+1}_{1,1}}
\]
for any $u_0\in \dot{B}^{N-p+1}_{1,1}(\htype)$ from {\rm Proposition~\ref{disp_linf_est} and (\ref{low_freq_dec_est1})}.
\end{rem}
\begin{proof}[Proof of Proposition~\ref{disp_linf_est}]
We note that $j<\log_2(1/\sqrt{|t|})$ is equivalent to $2^{(p-1)j}\le |t|^{-\frac{p-1}{2}}$. 
Therefore, it suffices to show that
\begin{equation}\label{disp_est_decom}
\|\Delta_je^{it\mathcal{L}}u_0\|_{L^{\infty}}
\le C 2^{(N-p+1)j}\min\{2^{(p-1)j},|t|^{-\frac{p-1}{2}}\}\|\Delta_ju_0\|_{L^1}
\end{equation}
for some constant $C>0$ which does not depend on $t$ and $j$. 
Because $\Phi_j=\Phi_j*\widetilde{\Phi}_j$ and $e^{it\mathcal{L}}(f*g)=(e^{it\mathcal{L}}f)*g=f*(e^{it\mathcal{L}}g)$ hold, 
we have 
\[
\begin{split}
\|\Delta_je^{it\mathcal{L}}u_0\|_{L^{\infty}}
=\|(e^{it\mathcal{L}}u_0)*\Phi_j*\widetilde{\Phi}_j)\|_{L^{\infty}}
=\|(u_0*\Phi_j)*(e^{it\mathcal{L}}\widetilde{\Phi}_j)\|_{L^{\infty}}. 
\end{split}
\]
Therefore, by the Young inequality, we get
\[
\|\Delta_je^{it\mathcal{L}}u_0\|_{L^{\infty}}
\lesssim \|u_0*\Phi_j\|_{L^1}\|e^{it\mathcal{L}}\widetilde{\Phi}_j\|_{L^{\infty}}
=\|e^{it\mathcal{L}}\widetilde{\Phi}_j\|_{L^{\infty}}\|\Delta_ju_0\|_{L^1}. 
\]
To obtain (\ref{disp_est_decom}), it suffices to show that
\begin{equation}\label{disp_est_decom_2}
\|\Delta_je^{it\mathcal{L}}\widetilde{\Phi}_j\|_{L^{\infty}}
\le C2^{(N-p+1)j}\min\{2^{(p-1)j},|t|^{-\frac{p-1}{2}}\}. 
\end{equation}
By using {\rm Proposition~\ref{inverse_formula}}, we obtain
\begin{equation}\label{st_linsol_fom}
\left(e^{it\mathcal{L}}\widetilde{\Phi}_j\right)(z,s)=\left(\frac{1}{2\pi}\right)^{d+p}\sum_{m\in \N}
\int_{\R^p}e^{-i\lambda s}\widehat{e^{it\mathcal{L}}\widetilde{\Phi}_j}(\lambda,m)L_m^{(d-1)}\left(\frac{|\lambda|}{2}|z|^2\right)|\lambda|^dd\lambda. 
\end{equation}
By the definition of the spherical Fourier transform {\rm (\ref{sp_ft_def})} 
with scaling argument, it holds
\[
\begin{split}
\widehat{\widetilde{\Phi}_j}(\lambda,m)
&=
\begin{pmatrix}
m+d-1\\
m
\end{pmatrix}^{-1}\disp\int_{\mathbb{R}^{2d+p}}e^{i\lambda s} 2^{Nj}\widetilde{\Phi}_0(2^jz,2^{2j}s)L^{(d-1)}_m\left(\dfrac{|\lambda|}{2}|z|^2\right)dzds\\
&=\begin{pmatrix}
m+d-1\\
m
\end{pmatrix}^{-1}\disp\int_{\mathbb{R}^{2d+p}}e^{i\lambda s} 2^{Nj}\widetilde{\Phi}_0(z,s)L^{(d-1)}_m\left(\dfrac{|2^{-2j}\lambda|}{2}|z|^2\right)2^{-2dj}dz2^{-2pj}ds\\
&=\widehat{\widetilde{\Phi}}_0(2^{-2j}\lambda,m). 
\end{split}
\]
Therefore, by (\ref{EV2}) and (\ref{st_linsol_fom}), we obtain
\[
\begin{split}
&\left(e^{it\mathcal{L}}\widetilde{\Phi}_j\right)(z,s)\\
&=\left(\frac{1}{2\pi}\right)^{d+p}\sum_{m\in \N}
\int_{\R^p}e^{-i\lambda 2^{2j}s}e^{i2^{2j}t(2m+d)|\lambda|}\widehat{\widetilde{\Phi}}_0(\lambda,m)L_m^{(d-1)}\left(\frac{|\lambda|}{2}|2^jz|^2\right)2^{2dj}|\lambda|^d2^{2pj}d\lambda\\
&=2^{Nj}\left(e^{i2^{2j}t\mathcal{L}}\widetilde{\Phi_0}\right)(2^jz,2^{2j}s). 
\end{split}
\]
Hence by Proposition~\ref{disp_hierro}, it holds
\[
\|e^{it\mathcal{L}}\widetilde{\Phi}_j\|_{L^{\infty}}
=2^{Nj}\|e^{i2^{2j}t\mathcal{L}}\widetilde{\Phi_0}\|_{L^{\infty}}
\lesssim 2^{Nj}\min\left\{1,|2^{2j}t|^{-\frac{p-1}{2}}\right\}, 
\]
where the implicit constant does not depend on $t$ and $j$. 
This implies (\ref{disp_est_decom_2}) and proof is complete. 
\end{proof}
To prove the Strichartz estimate, 
we first give the following lemma. 
\begin{lemm}\label{stri_lemm_1}
Let $d,p\in \N$ with $p\ge 2$ 
and $(q,r)$ is non-end point admissible pair. 
We put $\sigma =N(\frac{1}{2}-\frac{1}{r})-\frac{2}{q}$. 
Then, there exists $C>0$ such that
\begin{equation}\label{int_duam_est_stri}
\left\|\int_Ie^{-it\mathcal{L}}F(t)dt\right\|_{L^2}
\le C\|F\|_{L^{q'}_t\dot{B}^{\sigma}_{r',2}}
\end{equation}
holds for any interval $I\subset \R$ and 
$F\in L^{q}(\R;\dot{B}^{\sigma}_{r',2}(\htype))$, 
where
$q'$ and $r'$ denote the conjugate index of $q$ and $r$ respectively. 
\end{lemm}
\begin{proof}
Thanks to Mikowski's integral inequality, 
the desired estimate {\rm (\ref{int_duam_est_stri})} follows from 
\begin{equation}\label{int_duam_est_stri_j}
\left\|\int_I\Delta_je^{-it\mathcal{L}}F(t)dt\right\|_{L^2}
\le C2^{\sigma j}\|\Delta_jF\|_{L^{q'}_tL^{r'}_g}. 
\end{equation}
We put $\delta (r):=\frac{1}{2}-\frac{1}{r}$. 
For each $j\in \Z$, 
by the interpolation between {\rm (\ref{disp_est_decom})} 
and the unitarity of $e^{it\mathcal{L}}$, we have
\begin{equation}\label{intep_lr_est}
\|\Delta_je^{it\mathcal{L}}F\|_{L^r}
\lesssim 2^{2(N-p+1)\delta (r)j}\min\left\{2^{2(p-1)\delta (r)j},|t|^{-(p-1)\delta (r)}\right\}\|\Delta_jF\|_{L^{r'}},
\end{equation}
for $r\ge 2$, where implicit constant does not depend on $j$ and $t$. 
We note that
\[
\min\left\{2^{2(p-1)\delta (r)j},|t|^{-(p-1)\delta (r)}\right\}
\lesssim (2^{-2j}+|t|)^{-(p-1)\delta (r)}=:w(t). 
\]
Therefore, by the unitarity of $e^{it\mathcal{L}}$,  the H\"older inequality, and {\rm (\ref{intep_lr_est})} with $t=t_1-t_2$, 
we have 
\[
\begin{split}
\left\|\int_I\Delta_je^{-it\mathcal{L}}F(t)dt\right\|_{L^2}^2
&=\int_I\left(\Delta_jF(t_1),\int_I\Delta_je^{i(t_1-t_2)\mathcal{L}}F(t_2)dt_2\right)_{L^2}dt_1\\
&\le \int_I\|\Delta_jF(t_1)\|_{L^{r'}_g}\left(\int_I\left\|\Delta_je^{i(t_1-t_2)\mathcal{L}}F(t_2)\right\|_{L^r_g}dt_2\right)dt_1\\
&\lesssim 2^{2(N-p+1)\delta (r)j}\int_I\|\Delta_jF(t_1)\|_{L^{r'}_g}\left(\int_Iw(t_1-t_2)\left\|\Delta_jF(t_2)\right\|_{L^{r'}_g}dt_2\right)dt_1\\
&\lesssim 2^{2(N-p+1)\delta (r)j}\|\Delta_jF\|_{L^{q'}_tL^{r'}_g}\left\|\int_Iw(t-t_2)\left\|\Delta_jF(t_2)\right\|_{L^{r'}_g}dt_2\right\|_{L^q_{t}}. 
\end{split}
\]
To obtain {\rm (\ref{int_duam_est_stri_j})}, it suffices to show that
\begin{equation}\label{convol_t_est_y_HLS}
\left\|\int_Iw(t-t_2)\left\|\Delta_jF(t_2)\right\|_{L^{r'}_g}dt_2\right\|_{L^q_{t}}
\lesssim 2^{2\left\{(p-1)\delta (r)-\frac{2}{q}\right\}j}\|\Delta_jF\|_{L^{q'}_tL^{r'}_g}
\end{equation}
because
\[
2(N-p+1)\delta (r)+2\left\{(p-1)\delta (r)-\frac{2}{q}\right\}=2\sigma
\]
holds for any admissible pair $(q,r)$. \\

Now, we prove {\rm (\ref{convol_t_est_y_HLS})}. \\
\\
\noindent \underline{Case\ 1\ :\ $\frac{2}{q}<(p-1)\delta (r)$}\\

We put $\theta:=\frac{q}{2}$ and $a:=\theta (p-1)\delta (r)$. 
Because $|w(t)|^{\theta}=(2^{-2j}+|t|)^{-a}$ and
$a>1$, we have $w\in L^{\theta}(\R)$ and
\begin{equation}\label{w_theta_norm_cal}
\|w\|_{L^{\theta}}=\left(\int_{-\infty}^{\infty}(2^{-2j}+|t|)^{-a}dt\right)^{\frac{1}{\theta}}\lesssim 2^{\frac{2(a-1)}{\theta}j}.
\end{equation}
Furthermore, note that $1+\frac{1}{q}=\frac{1}{\theta}+\frac{1}{q'}$ holds. 
Therefore, by the Young inequality, we obtain
\[
\left\|w*\|\Delta_jf(\cdot)\|_{L^{r'}}\right\|_{L^q}\lesssim \|w\|_{L^{\theta}}\left\|\|\Delta_jf\|_{L^{r'}_g}\right\|_{L^{q'}_t}. 
\]
This estimate and {\rm (\ref{w_theta_norm_cal})} imply {\rm (\ref{convol_t_est_y_HLS})}. \\

\noindent \underline{Case\ 2\ :\ $\frac{2}{q}=(p-1)\delta (r)$ and $q>2$}\\

In this case, we cannot use the Young inequality as in Case\ 1
because $w\in L^{\frac{q}{2}}(\R)$ does not hold. 
However, we note that $w(t)\le |t|^{-\frac{2}{q}}$. 
Therefore, the Hardy-Littlewood-Sobolev inequality yields that
\begin{equation}\label{est_HLS_app}
\left\|w*\|\Delta_jf(\cdot)\|_{L^{r'}}\right\|_{L^q}
\le \left\||\cdot|^{-\frac{2}{q}}*\|\Delta_jf(\cdot)\|_{L^{r'}}\right\|_{L^q}
\lesssim \left\|\|\Delta_jf\|_{L^{r'}_g}\right\|_{L^{q'}_t}
\end{equation}
for $q>2$ and we get {\rm (\ref{convol_t_est_y_HLS})}. 
\end{proof}
\begin{rem}
For the case $\frac{2}{q}=(p-1)\delta (r)$ and $q=2$ 
$($namely, $(q,r)$ is end point$)$, 
we cannot obtain {\rm (\ref{est_HLS_app})} 
because the Hardy-Littlewood-Sobolev inequality 
fails in this case. 
\end{rem}
Here, we prove the Strichartz estimate. 
\begin{proof}[Proof of Theorem~\ref{H-Strichartz}]

We first prove
\begin{equation}\label{stri_besov_lins}
\|e^{it\mathcal{L}}u_0\|_{L^q_t\dot{B}^{-\sigma}_{r,2}}
\le C\|u_0\|_{L^2}
\end{equation}
for non-end point admissible pair $(q,r)$ and $\sigma =N(\frac{1}{2}-\frac{1}{r})-\frac{2}{q}$. 
By {\rm Lemma~\ref{stri_lemm_1}}, it holds that
\[
\left\|\int_{-\infty}^{\infty}e^{-it\mathcal{L}}F(t)dt\right\|_{L^2}\lesssim \|F\|_{L^{q'}_t\dot{B}^{\sigma}_{r',2}}. 
\]
Therefore, if we put $X=L^{q'}_t\dot{B}^{\sigma}_{r',2}$, then we have
\[
\begin{split}
\left|\langle F,e^{it\mathcal{L}}u_0\rangle_{X\times X^*}\right|
&=\left|\int_{-\infty}^{\infty}(F(t),e^{it\mathcal{L}}u_0)_{L^2}dt\right|\\
&=\left| \left(\int_{-\infty}^{\infty}e^{-it\mathcal{L}}F(t)dt,u_0\right)_{L^2}\right|\\
&\le \left\|\int_{-\infty}^{\infty}e^{-it\mathcal{L}}F(t)dt\right\|_{L^2}\|u_0\|_{L^2}
\lesssim \|F\|_{X}\|u_0\|_{L^2}. 
\end{split}
\]
This is equivalent to {\rm (\ref{stri_besov_lins})} 
because $X^*=L^q_t\dot{B}^{-\sigma}_{r,2}$ for $(q,r)\in [2,\infty]\times [2,\infty]$ and $\sigma\in \R$. 

Next, we prove
\begin{equation}\label{stri_besov_duam}
\left\|\int_0^te^{i(t-t')\mathcal{L}}F(t')dt'\right\|_{L^{q_1}_t\dot{B}^{-\sigma_1}_{r_1,2}}
\le C\|F\|_{L^{q_2'}_t\dot{B}^{\sigma_2}_{r_2',2}}
\end{equation}
for non-end point admissible pairs $(q_1,r_1)$ and $(q_2,r_2)$, and $\sigma_k =N(\frac{1}{2}-\frac{1}{r_k})-\frac{2}{q_k}$ $(k=1,2)$. 
By the similar argument as above, we obtain
\[
\left|\left\langle G,\int_0^te^{i(t-t')\mathcal{L}}F(t')dt'\right\rangle_{X\times X^*} \right|
\lesssim \|G\|_{X}\|F\|_{L^1_tL^2_g}
\]
for $X=L^{q_1'}_t\dot{B}^{\sigma_1}_{r_1',2}$. 
This is equivalent to {\rm (\ref{stri_besov_duam})} with $(q_2,r_2,\sigma_2)=(\infty,2,0)$. 
We also obtain
\[
\begin{split}
\left\|\int_0^te^{i(t-t')\mathcal{L}}F(t')dt'\right\|_{L^2_g}
&=\left\|\int_0^te^{-it'\mathcal{L}}F(t')dt'\right\|_{L^2_g}
\lesssim \|F\|_{L^{q_2'}_r\dot{B}^{\sigma_2}_{r_2',2}}
\end{split}
\]
for any $t\in \R$ by {\rm (\ref{stri_lemm_1})} 
and the unitarity of $e^{it\mathcal {L}}$. 
This implies {\rm (\ref{stri_besov_duam})} with $(q_1,r_1,\sigma_1)=(\infty,2,0)$. 
By multiplying {\rm (\ref{intep_lr_est})} by $2^{-\sigma j}$, taking the $l^2$-summation, 
and similar argument as in the proof of {\rm (\ref{convol_t_est_y_HLS})}, we have
\[
\left\|\int_0^te^{i(t-t')\mathcal{L}}F(t')dt'\right\|_{L^q_t\dot{B}^{-\sigma}_{r,2}}
\lesssim \|F\|_{L^{q'}_t\dot{B}^{\sigma}_{r',2}}. 
\]
This is {\rm (\ref{stri_besov_duam})} with $(q_1,r_1,\sigma_1)=(q_2,r_2,\sigma_2)\ (=(q,r,\sigma))$. 
By interpolation between the cases $(q_2,r_2,\sigma_2)=(\infty,2,0)$ and  $(q_2,r_2,\sigma_2)=(q_1,r_1,\sigma_1)$, 
and interpolation between the cases $(q_1,r_1,\sigma_1)=(\infty,2,0)$ and $(q_1,r_1,\sigma_1)=(q_2,r_2,\sigma_2)$, 
we obtain {\rm (\ref{stri_besov_duam})} for general cases
(For the interpolation of Besov spaces on $\htype$, see \cite{Liu}).

Finally, we prove the Strichartz estimates {\rm (\ref{stri_hom_est})} and {\rm (\ref{stri_inhom_est})} 
by using {\rm (\ref{stri_besov_lins})} and {\rm (\ref{stri_besov_duam})}. 
We prove only {\rm (\ref{stri_inhom_est})} since the proof of {\rm (\ref{stri_hom_est})} is similar and more simpler. 
By putting $G(t)=(-\mathcal{L})^{\frac{\sigma_1}{2}}F(t)$, 
we have
\[
\begin{split}
\left\|\int_0^te^{-i(t-t')\mathcal{L}}F(t')dt'\right\|_{L^{q_1}_tL^{r_1}_g}
&=\left\|\int_0^te^{-i(t-t')\mathcal{L}}G(t')dt'\right\|_{L^{q_1}_t\dot{W}^{-\sigma_1,r_1}_g}. 
\end{split}
\]
Therefore, by the second inequality in (\ref{besov_sobo_embed}) 
and {\rm (\ref{stri_besov_duam})}, we obtain
\[
\begin{split}
\left\|\int_0^te^{-i(t-t')\mathcal{L}}F(t')dt'\right\|_{L^{q_1}_tL^{r_1}_g}
&\lesssim \left\|\int_0^te^{-i(t-t')\mathcal{L}}G(t')dt'\right\|_{L^{q_1}_t\dot{B}^{-\sigma_1}_{r_1,2}}
\lesssim \left\|G\right\|_{L^{q_2'}_t\dot{B}^{\sigma_2}_{r_2',2}}. 
\end{split}
\]
We note that
\[
 \left\|G\right\|_{L^{q_2'}_t\dot{B}^{\sigma_2}_{r_2',2}}
 = \left\|F\right\|_{L^{q_2'}_t\dot{B}^{\sigma_1+\sigma_2}_{r_2',2}}. 
\]
Thus, by the first inequality in (\ref{besov_sobo_embed}), we get {\rm (\ref{stri_inhom_est})}. 
\end{proof}
\section{The proof of well-posedness}\label{sec_wp}
\noindent

In this section, we prove the well-posedness of 
the nonlinear Schr\"odinger equation (Theorem~\ref{th1} and~\ref{th_cri}). 
We construct the solution to (\ref{HNLS}) by using the iteration argument. 
For $0<T\le \infty$, 
we define the solution space $X_T^s$ by
\[
X_T^s:=L^{\infty}([0,T);H^s(\htype))\cap L^q((0,T);W^{s-s_*,r}(\htype))
\]
with the norm
\[
\begin{split}
\|u\|_{X^s_T}
&:=
\|u\|_{L^{\infty}([0,T);H^s(\htype))}+\|u\|_{L^q((0,T);W^{s-s_*,r}(\htype))}\\
&=\|\ee_{[0,T)}(t)({\rm Id}+\mathcal{L})^{\frac{s}{2}}u\|_{L^{\infty}_tL^2_g}
+\|\ee_{(0,T)}(t)({\rm Id}+\mathcal{L})^{\frac{s-s_*}{2}}u\|_{L^q_tL^r_g}, 
\end{split}
\]
where $(q,r)$ is a admissible pair, which will be chosen later. 
For $u_0\in H^s(\htype)$, 
we define the functional $\Phi_{u_0}$ on $X^s_T$ by
\[
\Phi_{u_0}[u](t):=e^{it\mathcal{L}}u_0-i\mu \int_0^te^{i(t-t')\mathcal{L}}\left(|u(t')|^{\alpha -1}u(t')\right)dt'.
\]
We remind that $s_c=\frac{N}{2}-\frac{2}{\alpha -1}$ and $s_*=\max\{\frac{N-p+1}{2},\frac{N-2}{2},s_c\}$. 
To get the estimate for $\Phi_{u_0}[u]$, we first prove the following. 
\begin{lemm}\label{adm_exist}
Let $d,p\in \N$ with $p\ge 2$ and $s_*< s<\frac{N}{2}$. 
Then, there exists 
non-end point admissible pair $(q,r)\in (2,\infty)\times (2,\infty)$ such that 
$s_*=N\left(\frac{1}{2}-\frac{1}{r}\right)-\frac{2}{q}$, 
$q>\alpha -1$, and $s-s_* > \frac{N}{r}$ hold. 
\end{lemm}
\begin{proof}
We choose $\delta > 0$ small enough
satisfying $\delta< s-s_*$, 
and put
\begin{equation}\label{adm_q_r_def}
q=\frac{4}{N-2s+2\delta},\ \ \ \ r=\frac{N}{s-s_*-\delta}. 
\end{equation}
Then, we have $q> 2$ by $s> \frac{N-2}{2}$.  
We also note that $r<\infty$ and $s-s_* >\frac{N}{r}$. 
Furthermore, 
direct calculation shows that
\[
s_*=N\left(\frac{1}{2}-\frac{1}{r}\right)-\frac{2}{q}
\]
and therefore we have
\[
\begin{split}
(p-1)\left(\frac{1}{2}-\frac{1}{r}\right)-\frac{2}{q}
&=(p-1)\left(\frac{1}{2}-\frac{1}{r}\right)-\left\{N\left(\frac{1}{2}-\frac{1}{r}\right)-s_*\right\}\\
&> \frac{N-p+1}{r}>0
\end{split}
\]
by $s_*>\frac{N-p+1}{2}$. 
This says that $(q,r)$ is non-end point admissible pair. 
The inequality $q>\alpha -1$ follows from $s>s_c$. 
\end{proof}
The following lemma will be used to treat critical cases. 
\begin{lemm}\label{adm_exist_2}
Let $d,p\in \N$ with $p\ge 2$. 
\begin{itemize}
\item[{\rm (i)}]Assume $p>3$ and $1<\alpha <3$, 
Then $(q,r)=(2,\infty)$ is 
non-end point admissible pair such that $s_*=\frac{N}{2}-\frac{2}{q}$ and $q>\alpha -1$ hold. 
\item[{\rm (ii)}]Assume $p=2$ and $1<\alpha <5$.  
Then $(q,r)=(4,\infty)$ is 
non-end point admissible pair such that $s_*=\frac{N}{2}-\frac{2}{q}$ and $q>\alpha -1$ hold. 
\item[{\rm (iii)}] Assume one of 
\[
{\rm (a)}\ p=2,\ \alpha \ge 5\ \ \ {\rm (b)}\ p=3,\ \alpha>3\ \ \ {\rm (c)}\ p=4,5,6,\cdots,\ \alpha \ge 3
\]
be satisfied. Then $(q,r)=(\alpha-1,\infty)$ is 
non-end point admissible pair such that $s_*=\frac{N}{2}-\frac{2}{q}$ holds. 
\end{itemize}
\end{lemm}
\begin{proof}
We choose $(q,r)$ as in {\rm (\ref{adm_q_r_def})} with $s=s_*$ and $\delta =0$. 
Then, $r=\infty$ and we can check that 
\[
s_*=\frac{N}{2}-\frac{2}{q}. 
\]
Furthermore, we have
\[
q=
\begin{cases}
2&{\rm if}\ p>3,\ 1<\alpha <3,\\
4&{\rm if}\ p=2,\ 1<\alpha<5,\\
\alpha -1&{\rm otherwise}.
\end{cases}
\]
because
\[
s_*=
\begin{cases}
\frac{N-2}{2}&{\rm if}\ p>3,\ 1<\alpha <3,\\
\frac{N-1}{2}&{\rm if}\ p=2,\ 1<\alpha<5,\\
s_c&{\rm otherwise}.
\end{cases}
\]
Clearly, it hold
\[
q>\alpha -1,\ \ 
\frac{2}{q}<(p-1)\left(\frac{1}{2}-\frac{1}{r}\right)
\]
for $p>3$, $1<\alpha <3$, and $(q,r)=(2,\infty)$. 
We also have 
\[
q>\max\{2,\alpha-1\},\ \ 
\frac{2}{q}=(p-1)\left(\frac{1}{2}-\frac{1}{r}\right)
\]
for $p=2$, $1<\alpha <5$, and $(q,r)=(4,\infty)$. 
Thus, we obtain the conclusion in {\rm (i)} and {\rm (ii)}. 

Finally, we assume one of {\rm (a)}, {\rm (b)}, and {\rm (c)} in {\rm (iii)} be satisfied.  
Then, we can check that
\[
q>2\ \ {\rm or}\ \ \frac{2}{q}<(p-1)\left(\frac{1}{2}-\frac{1}{r}\right)
\]
for $(q,r)=(\alpha-1,\infty)$. 
Therefore, we obtain the conclusion in {\rm (iii)}. 
\end{proof}
\begin{prop}\label{duam_est_1_prop}
Let $d,p\in \N$ with $p\ge 2$, $T>0$, $\alpha >1$, and $s_*\le s<\frac{N}{2}$. 
Assume that $(q,r)$ is the admissible pair as in 
{\rm Lemma~\ref{adm_exist}} when $s>s_*$, and as in {\rm Lemma~\ref{adm_exist_2}} when $s=s_*$. 
We also assume $\alpha \ge \lceil s\rceil +1$ if $\alpha$ is not an odd integer. 
Then, 
\begin{equation}\label{duam_single_est}
\left\|\int_0^te^{i(t-t')\mathcal{L}}\left(|u(t')|^{\alpha -1}u(t')\right)dt'\right\|_{X^s_T}
\le CT^{1-\frac{\alpha-1}{q}}\|u\|_{X^s_T}
\end{equation}
holds for any $u\in X^s_T$ and
\begin{equation}\label{duam_diff_est}
\begin{split}
&\left\|\int_0^te^{i(t-t')\mathcal{L}}\left(|u(t')|^{\alpha -1}u(t')-|v(t')|^{\alpha -1}v(t')\right)dt'\right\|_{X^s_T}\\
&\ \ \ \ \ \ \ \ \le CT^{1-\frac{\alpha-1}{q}}(\|u\|_{X^s_T}^{\alpha -1}+\|v\|_{X^s_T}^{\alpha -1})\|u-v\|_{X^s_T}
\end{split}
\end{equation}
holds for any $u,v\in X^s_T$. 
In particular, we can choose $T=\infty$ when $s=s_*=s_c$. 
\end{prop}
\begin{proof}
We first prove {\rm (\ref{duam_single_est}) }. 
By The Strichartz estimate {\rm (\ref{Stri_inhom_sob_duam})}
with $(\sigma_1,q_1,r_1)=(\sigma_2,q_2,r_2)=(0,\infty,2)$, 
we have
\begin{equation}\label{duam_hsnorm_est}
\left\|\int_0^te^{i(t-t')\mathcal{L}}\left(|u(t')|^{\alpha -1}u(t')\right)dt'\right\|_{L^{\infty}([0,T);H^s(\htype))}
\lesssim \|\ee_{[0,T)}(t)|u|^{\alpha-1}u\|_{L^1_tH^s_g}. 
\end{equation}
On the other hand, 
by {\rm (\ref{Stri_inhom_sob_duam})} 
with $(\sigma_1,q_1,r_1)=(s_*,q,r)$, $(\sigma_2,q_2,r_2)=(0,\infty,2)$, 
we have
\[
\left\|\int_0^te^{i(t-t')\mathcal{L}}\left(|u(t')|^{\alpha -1}u(t')\right)dt'\right\|_{L^{q}((0,T);W^{s-s_*,r}(\htype))}
\lesssim \|\ee_{[0,T)}(t)|u|^{\alpha-1}u\|_{L^1_tH^s_g}. 
\]
Therefore, it suffices to show that
\begin{equation}\label{duam_est_des}
\|\ee_{[0,T)}(t)|u|^{\alpha-1}u\|_{L^1_tH^s_g}
\lesssim T^{1-\frac{\alpha-1}{q}}\|u\|_{X^s_T}.
\end{equation}

Now, we show (\ref{duam_est_des}).
By using the generalized chain rule ({\rm(\ref{FLR_ineq})} and Remark~\ref{IFLR} ), we obtain
\[
\begin{split}
\|\ee_{[0,T)}(t)|u|^{\alpha-1}u\|_{L^1_tH^s_g}
&\lesssim \|\ee_{[0,T)}(t)|u|^{\alpha-1}\|_{L^1_tL^{\infty}_g}\|\ee_{[0,T)}(t)u\|_{L^{\infty}_tH^s_g}\\
&\lesssim \|\ee_{[0,T)}(t)|u|^{\alpha-1}\|_{L^1_tL^{\infty}_g}\|u\|_{X^s_T}.
\end{split}
\]
Furthermore, by the H\"older inequality, we can see that
\[
\begin{split}
\|\ee_{[0,T)}(t)|u|^{\alpha-1}\|_{L^1_tL^{\infty}_g}
&\le \|\ee_{[0,T)}(t)u\|_{L^{\alpha-1}_tL^{\infty}_g}^{\alpha-1}\\
&\le \left(\|\ee_{[0,T)}(t)\|_{L_t^{\frac{(\alpha -1)q}{q-(\alpha -1)}}}
\|\ee_{[0,T)}(t)u\|_{L^q_tL^{\infty}_g}\right)^{\alpha-1}\\
&=T^{1-\frac{\alpha-1}{q}}\|\ee_{[0,T)}(t)u\|_{L^q_tL^{\infty}_g}^{\alpha-1},
\end{split}
\]
where, we used $q\ge \alpha -1$. 
If $r=\infty$ (namely, $s=s_*$ and $(q,r)$ is as in {\rm Lemma~\ref{adm_exist_2}}), 
then proof is completed 
because $X^{s_*}_T\hookrightarrow L^{q}((0,T);L^{\infty}(\htype))$ holds. 
If $r<\infty$  (namely, $s>s_*$ and $(q,r)$ is as in {\rm Lemma~\ref{adm_exist}}), 
by using the Sobolev embedding $W^{s-s_*,r}(\htype)\hookrightarrow L^{\infty}(\htype)$ 
with $s-s_*>\frac{N}{r}$, 
we get
\[
\|\ee_{[0,T]}(t)u\|_{L^q_tL^{\infty}_g}
\lesssim \|\ee_{[0,T]}(t)u\|_{L^q_tW^{s-s_*,r}_g}
\le \|u\|_{X^s_T}.
\]
As a result, we have (\ref{duam_est_des}). 
We note that $q=\alpha-1$ holds when $s=s_*=s_c$. 
Therefore, we can choose $T=\infty$ when $s=s_*=s_c$.

Similarly as the proof of {\rm (\ref{duam_single_est}) }, 
by applying {\rm (\ref{FLR_diff_ineq})}, we obtain
\[
\begin{split}
&\left\|\int_0^te^{i(t-t')\mathcal{L}}\left(|u(t')|^{\alpha -1}u(t')-|v(t')|^{\alpha -1}v(t')\right)dt'\right\|_{X^s_T}\\
&\lesssim \left(\|\ee_{[0,T)}(t)u\|_{L_t^{\alpha-1}L_g^{\infty}}^{\alpha-1}+\|\ee_{[0,T)}(t)v\|_{L_t^{\alpha-1}L_g^{\infty}}^{\alpha-1}\right)\|\ee_{[0,T)}(t)(u-v)\|_{L^{\infty}_tH^s_g}\\
&\ \ \ \ \ \ +\left(\|\ee_{[0,T)}(t)u\|_{L_t^{\alpha-1}L_g^{\infty}}^{\alpha-2}+\|\ee_{[0,T)}(t)v\|_{L_t^{\alpha-1}L_g^{\infty}}^{\alpha-1}\right)\\
&\ \ \ \ \ \ \ \ \ \ \ \ \times \left(\|\ee_{[0,T)}(t)u\|_{L^{\infty}_tH^s_g}+\|\ee_{[0,T)}(t)v\|_{L^{\infty}_tH^s_g}\right)\|\ee_{[0,T)}(t)(u-v)\|_{L^{\alpha-1}_tL^{\infty}_g}. 
\end{split}
\]
Therefore, we have {\rm (\ref{duam_diff_est})} by the same argument as above. 
\end{proof}
\begin{rem}
We can also obtain
\begin{equation}\label{duam_diff_est_reg0}
\begin{split}
&\left\|\int_0^te^{i(t-t')\mathcal{L}}\left(|u(t')|^{\alpha -1}u(t')-|v(t')|^{\alpha -1}v(t')\right)dt'\right\|_{X^0_T}\\
&\ \ \ \ \ \ \ \ \le CT^{1-\frac{\alpha-1}{q}}(\|u\|_{X^s_T}^{\alpha -1}+\|v\|_{X^s_T}^{\alpha -1})\|u-v\|_{X^0_T}
\end{split}
\end{equation}
by using $||u|^{\alpha-1}u-|v|^{\alpha-1}v|\lesssim (|u|^{\alpha-1}+|v|^{\alpha-1})|u-v|$ 
instead of {\rm (\ref{FLR_diff_ineq})} in the proof of {\rm Proposition~\ref{duam_est_1_prop}}. 
\end{rem}
Now, we give the proof of the local well-posedness. 
\begin{proof}[Proof of Theorems~\ref{th1} and ~\ref{th_cri}]
We assume the assumption (\ref{smooth_cond_nonl}). 
Under the weaker condition $\alpha \ge \lceil s\rceil$, see Remark~\ref{ex_uni_rem}.
By {\rm (\ref{Stri_inhom_sob_lin})}, there exists $C_1>0$ such that
 \[
 \|e^{it\mathcal{L}}u_0\|_{X^s_T}
 \le \|e^{it\mathcal{L}}u_0\|_{L^{\infty}_tH^s_g}+ \|e^{it\mathcal{L}}u_0\|_{L^{q}_tW^{s-s_*,r}_g}
 \le C_1\|u_0\|_{H^s}
 \]
holds for  any $u_0\in H^s(\htype)$.  
Let $\rho>0$ and $u_0\in H^s(\htype)$ with $\|u_0\|_{H^s}\le \rho$. 
We set 
\[
X^s_T(\rho):=\{u\in X^s_T|\ \|u\|_{X^s_T}\le 2C_1\rho\}
\]
and define the metric $d$ on $X^s_T(\rho)$ by
\[
d(u,v):=\|u-v\|_{X^s_T}. 
\] 
Then, $(X^s_T(\rho), d)$ becomes a complete metric space. 
By Proposition~\ref{duam_est_1_prop}, 
we have
\[
\begin{split}
\|\Phi_{u_0}[u]\|_{X^s_T}
&\le \|e^{it\mathcal{L}}u_0\|_{X^s_T}+|\mu|\left\|\int_0^te^{i(t-t')\mathcal{L}}\left(|u(t')|^{\alpha -1}u(t')\right)dt'\right\|_{X^s_T}\\
&\le C_1\|u_0\|_{H^{s}}+CT^{1-\frac{\alpha-1}{q}}\|u\|_{X^s_T}^{\alpha}
\le (1+CT^{1-\frac{\alpha-1}{q}}2^{\alpha}C_1^{\alpha -1}\rho^{\alpha -1})C_1\rho. 
\end{split}
\]
for any $u\in X^s_T(\rho)$ and
\[
\begin{split}
d(\Phi_{u_0}[u],\Phi_{u_0}[v])
&\le CT^{1-\frac{\alpha-1}{q}}(\|u\|_{X^s_T}^{\alpha-1}+\|v\|_{X^s_T}^{\alpha-1})\|u-v\|_{X^s_T}\\
&\le CT^{1-\frac{\alpha-1}{q}}2^{\alpha}C_1^{\alpha-1}\rho^{\alpha-1}d(u,v)
\end{split}
\]
for any $u,v\in X^s_T(\rho)$. 
Therefore, if we choose $T>0$ as 
\[
T^{1-\frac{\alpha-1}{q}}<\frac{1}{C2^{\alpha}C_1^{\alpha-1}\rho^{\alpha-1}}, 
\]
then $\Phi_{u_0}$ is a contraction map on $X^s_T(\rho)$ and 
we can get the unique solution $u\in X^s_T(\rho)$ to $u=\Phi_{u_0}[u]$ on $[0,T]$. 
If $s=s_c$ (then $q=\alpha-1$), we choose $\rho>0$ as 
\[
\rho^{\alpha-1}
<\frac{1}{C2^{\alpha}C_1^{\alpha-1}}. 
\]
Then, we have that $\Phi_{u_0}$ is a contraction map on $X^s_{\infty}(\rho)$. 

The uniqueness in $X^s_T$ and the continuous dependence of the solution map on initial data 
follow from the estimate
\[
\begin{split}
d(u,v)&=d(\Phi_{u_0}[u],\Phi_{v_0}[v])\\
&\le C\|u_0-v_0\|_{H^s}+CT^{1-\frac{\alpha-1}{q}}\left(\|u\|_{X^s_T}^{\alpha-1}+\|v\|_{X^s_T}^{\alpha-1}\right)d(u,v)
\end{split}
\]
for solutions $u$ and $v$ with initial data $u_0$ and $v_0$ respectively. 
\end{proof}
\begin{rem}\label{ex_uni_rem}
If we use the metric
\[
d_0(u,v):=\|u-v\|_{X^0_T}
=\|u-v\|_{L^{\infty}([0,T);L^2(\mathbb{H}^d_p))}
+\|u-v\|_{L^{q}([0,T);W^{-s_*,r}(\mathbb{H}^d_p))}
\]
and the estimate {\rm (\ref{duam_diff_est_reg0})} 
in the Proof of Theorema~\ref{th1} and ~\ref{th_cri}, then
we can apply the contraction mapping principle 
for the complete metric space $(X^s_T,d_0)$ 
without using {\rm (\ref{FLR_diff_ineq})}. 
As consequence, 
we can obtain the existence of unique solution 
to {\rm (\ref{HNLS})} under the weaker assumption 
$\alpha \ge \lceil s\rceil$ than {\rm (\ref{smooth_cond_nonl})} 
if $\alpha$ is not an odd integer. 
The continuity of the data-to-solution map 
from $H^s(\htype)$ to $C([0,T];H^{s-\eta}(\htype))$ 
is also obtained by using the estimate
\[
\|u\|_{H^{s-\eta}(\htype)}\le C\|u\|_{H^s(\htype)}^{1-\frac{\eta}{s}}\|u\|_{L^2(\htype)}^{\frac{\eta}{s}}.
\]
But we omit the details $($see, {\rm \cite{Dinh17}}$)$. 
The above interpolation estimate can be found in {\rm \cite{CRN01} Proposition~31}. 
\end{rem}
\section*{Acknowledgements}
This work was supported by
JSPS KAKENHI Grant Number JP 21K03333.

\noindent
Hiroyuki HIRAYAMA\\
Faculty of Education, University of Miyazaki,\\
1-1, Gakuenkibanadai-nishi, Miyazaki, 889-2192 Japan\\
e-mail: h.hirayama@cc.miyazaki-u.ac.jp\\
\\
\noindent
Yasuyuki OKA\\
School of Liberal Arts and Sciences,  Daido university, \\
10-3 Takiharu-cho, Minami-ku, Nagoya 457-8530 Japan \\
e-mail: y-oka@daido-it.ac.jp

\end{document}